\documentclass[11pt,reqno]{amsart}

\usepackage[left=2.2cm,right=2.2cm,top=1.5cm,bottom=1.5cm,includeheadfoot]{geometry}
\usepackage{amssymb}
\usepackage{amscd}
\usepackage{amsfonts}
\usepackage{mathrsfs,mathscinet}
\usepackage{setspace}
\usepackage{version}
\usepackage{mathtools}
\usepackage{enumerate}
\usepackage[english]{babel}
\usepackage[dvipsnames]{xcolor}
\numberwithin{equation}{section}

\newtheorem{theorem}{Theorem}[section]
\newtheorem{proposition}[theorem]{Proposition}
\newtheorem{corollary}[theorem]{Corollary}
\newtheorem{lemma}[theorem]{Lemma}

\theoremstyle{definition}
\newtheorem{assumption}[theorem]{Assumption}
\newtheorem{remark}[theorem]{Remark}
\newtheorem{observation}[theorem]{Observation}
\newtheorem{example}[theorem]{Example}
\newtheorem{definition}[theorem]{Definition}

\newcommand{\R}{\mathbb{R}}      
      
\newcommand{\QW}{\mathbb{Q}}
\newcommand{\E}{\mathbb{E}}
\newcommand{\tn}{\textnormal}
\newcommand{\ind}{\mathbf{1}}
\newcommand{\Norm}{\|\cdot\|}
\newcommand{\N}{\mathbb{N}}
\newcommand{\dom}{\textnormal{dom}}
\newcommand{\cl}{\textnormal{cl}}
\newcommand{\eps}{\varepsilon}

\newcommand{\Linfty}{L^{\infty}(\mathfrak P)}

\newcommand{\CN}{\mathcal N}
\newcommand{\CF}{\mathcal F}

\newcommand{\CX}{\mathcal X}

\newcommand{\CP}{\mathfrak P}
\newcommand{\CC}{\mathcal C}
\newcommand{\CL}{\mathcal L}
\newcommand{\CH}{\mathcal H}
\newcommand{\CB}{\mathcal B}
\newcommand{\CY}{{\mathcal Y}}
\newcommand{\CI}{{\mathcal I}}
\newcommand{\CQ}{\mathfrak Q}
\newcommand{\w}{\widehat}
\newcommand{\mf}{\mathfrak}
\newcommand{\ca}{\mathbf{ca}}

\newcommand{\mbf}{\mathbf}
\newcommand{\peq}{\preceq}
\newcommand{\tle}{\trianglelefteq}
\newcommand{\Max}{{\rm Max}}

\newcommand{\Om}{\Omega}
\newcommand{\om}{\omega}
\newcommand{\si}{\sigma}
\newcommand{\la}{\lambda}
\newcommand{\al}{\alpha}

\newcommand{\ga}{\gamma}

\newcommand{\ep}{\varepsilon}
\renewcommand{\P}{\mathbb P}
\newcommand{\Q}{\mathbb Q}

\title[Separability vs.\ robustness of Orlicz spaces]{Separability vs.\ robustness of Orlicz spaces: financial and economic perspectives}

\author{Felix-Benedikt Liebrich}
\address{Institute of Actuarial and Financial Mathematics \& House of Insurance,\
Leibniz University Hannover, Germany}
\email{felix.liebrich@insurance.uni-hannover.de}

\author{Max Nendel}
\address{Center for Mathematical Economics, Bielefeld University, Germany}
\email{max.nendel@uni-bielefeld.de}

\thanks{The authors thank Julian H\"olzermann, Fabio Maccheroni, Frank Riedel, and Johannes Wiesel for helpful comments and suggestions related to this work. The second author gratefully acknowledges financial support of the German Research Foundation via CRC 1283.}

\begin{document}

\parindent 0em \noindent
\thispagestyle{empty}

\begin{abstract}
We investigate robust Orlicz spaces as a generalisation of robust $L^p$-spaces. Two constructions of such spaces are distinguished, a top-down approach and a bottom-up approach.\ We show that separability of robust Orlicz spaces or their subspaces has very strong implications in terms of the dominatedness of the set of priors and the lack of order completeness. Our results have subtle implications for the field of robust finance. For instance, norm closures of bounded continuous functions with respect to the worst-case $L^p$-norm, as considered in the $G$-framework, lead to spaces which are lattice isomorphic to a sublattice of a classical $L^1$-space lacking, however, any form of order completeness.\ We further show that the topological spanning power of options is always limited under nondominated uncertainty.   


\end{abstract}

\maketitle


\section{Introduction}

Since the beginning of this century, the simultaneous consideration of families of prior distributions instead of a single probability measure has become of fundamental importance for the risk assessment of financial positions. In this context, one often speaks of \textit{model uncertainty} or \textit{ambiguity}, where the uncertainty is modeled by a set $\CP$ of probability measures. Especially after the subprime mortgage crisis, the desire for mathematical models based on \textit{nondominated} families of priors arose: no single reference probability measure can be chosen which determines whether an event is deemed certain or negligible. Beside the seminal contribution of \cite{BN}, the to date most prominent example of a model involving nondominated uncertainty is a Brownian motion with uncertain volatility, the so-called $G$-Brownian motion, cf.\ \cite{PengG}.\footnote{~More precisely, the resulting measures are mutually singular.} The latter is intimately related to the theory of second-order backward stochastic differential equations, cf.\ \cite{MR2319056}. An extensive strand of literature has formed around this model.

On another note, there has been renewed interest in the role of \textit{Orlicz spaces} in mathematical finance in the past few years. They have appeared, e.g., as canonical model spaces for risk measures, premium principles, and utility maximisation problems, see \cite{OrliczQuant,Cerny,CheriditoLi, FatouRep} and many others.

{The present manuscript investigates \textit{robust} Orlicz spaces in a setting of potentially nondominated sets $\CP$ of probability measures and the role of separability in financial and economic applications.} \textit{A priori}, {the construction of robust Orlicz spaces} can follow two conceivable paths which often lead to the same result if a probability space is underlying, {i.e., $\CP=\{\P\}$.} The classical $L^p(\P)$-space, for $p\in[1,\infty)$, shall serve as illustration. It can be obtained either by a \textit{top-down} approach, considering the maximal set of all equivalence classes of real-valued measurable functions with finite norm $\Norm_{L^p(\P)}:=\E_\P[|\,\cdot\,|^p]^{1/p}$; or, equivalently, by proceeding in a \textit{bottom-up} manner, closing a smaller test space of, say, bounded random variables w.r.t.\ the norm $\Norm_{L^p(\P)}$. If the underlying state space is topological, one has even more degrees of freedom and may select suitable spaces of continuous functions as a test space. Both approaches lead to a \textit{Banach lattice} which naturally carries the $\P$-almost-sure order, and this turns out to be $L^p(\P)$ in both cases. Morally speaking, the reason for this equivalence is  that $\Norm_{L^p(\P)}$ is not very robust and rather insensitive to the tail behaviour of a given random variable. In general, the two paths tend to diverge substantially for \textit{(robust) Orlicz spaces}, {while both turn out to have their economic merits. The present note may be understood as a comparison of the two approaches against the backdrop of financial applications.}

Throughout, we consider a fixed measurable space $(\Omega,\CF)$, a nonempty set of probability measures (priors) $\CP$ on $(\Om,\CF)$, and a family $\Phi=(\phi_\P)_{\P\in\CP}$ of Orlicz functions, {possibly varying with the prior.} On the quotient space $L^0(\CP)$ of all real-valued measurable functions on $(\Omega,\CF)$ up to $\CP$-quasi-sure ($\CP$-q.s.) equality, we consider the robust Luxemburg norm
 \begin{equation}\label{eq:introduction}
  \|X\|_{L^\Phi(\CP)}:=\sup_{\P\in \CP}\|X\|_{L^{\phi_\P}(\P)}\in [0,\infty],\quad \text{for }X\in L^0(\CP),
 \end{equation}
where $\|\cdot\|_{L^{\phi_\P}(\P)}$ denotes the Luxemburg seminorm for $\phi_\P$ under the probability measure $\P\in \CP$. The (top-down) robust Orlicz space $L^\Phi(\CP)$ is then defined to be the space of all $X\in L^0(\CP)$ with $\|X\|_{L^\Phi(\CP)}<\infty$. 

{In the field of robust finance, special cases of robust Orlicz spaces have been studied, e.g., in the $G$-Framework \cite{MR2974730,MR2925572} and over general measurable spaces \cite{surplus,KupSvi,Liebrich,Maggisetal,Owari}. However, defining $L^p$-spaces by a worst-case approach over a set of priors instead of a single prior stems back at least to \cite{MR182874}. Noteworthy are \cite{MR856304,MR1038172}, where a general Orlicz function is considered instead of the $L^p$-case. We essentially work in their framework. A more detailed discussion of these references is provided in Remark~\ref{rem:RoyChakraborty}.

Another top-down approach to robust Orlicz spaces that appears in the literature -- cf.\ \cite{Nutz,MR2842089} and \cite[Example 2.6]{Viability} -- is given by the space
 \begin{align*}\mf L^\Phi(\CP)&:=\big\{X\in L^0(\CP)\,\big|\, \forall\,\P\in\CP\,\exists\,\alpha>0:~\E_\P[\phi_\P(\alpha|X|)]<\infty\big\}\\
 &=\big\{X\in L^0(\CP)\,\big|\, \forall\,\P\in\CP:~\|X\|_{L^{\phi_\P}(\P)}<\infty\big\},\end{align*}
 the ``intersection'' of the individual Orlicz spaces. We discuss this construction in Section~\ref{sec:alternative} and show in Proposition~\ref{prop:equality} and Theorem~\ref{thm:equality} that, in many situations,}
 \[\mf L^\Phi(\CP)=L^\Phi(\CP).\]
 These uniform-boundedness-type results prove the equivalence of both constructions in terms of the extension of the resulting spaces in $L^0(\CP)$.

 \bigskip 
  
\section{Main results and related literature}\label{main results and related}

{For better orientation of the reader, we summarise below the implications of our results for the field of robust finance and discuss further related literature. A more detailed discussion can be found in Section~\ref{sec.applications}.}
 
\subsection{Function spaces in the $G$-Framework}
In order to combine analytic tractability and uncertainty modelled by nondominated priors, a rich strand of literature has pursued the bottom-up construction of robust Orlicz spaces (though not in the generality of our definition above). For $G$-Brownian motion, in particular, closures $\CC^p$ of the space $C_b$ of bounded continuous functions under robust $L^p$-norms $\Norm_{L^p(\CP)}$ for a nonempty set of priors $\CP$ have become a frequent choice for commodity spaces or spaces of contingent claims, cf.\ \cite{MR3816550,MR3048198,hoelzermann2}. We shall cover such closures in Section~\ref{sec:continuous}. The analytic properties of these spaces have been studied in, e.g., \cite{Beissner,Peng}, and a complete stochastic calculus has been developed for them, cf.\ \cite{peng2019nonlinear}. {There are also preliminary studies of their order structure against the backdrop of financial applications, cf.\ \cite{BionNadal} and the discussion in \cite{Cohen}.}
Apart from this, little is known about $\CC^p$ as Banach lattice. We shall fill this gap {and prove in Section~\ref{sec:continuous} that $\CC^p$ is usually separable. As such, it illustrates a general conflict between separability and nondominatedness of the underlying priors -- or robustness of the function space in question -- that the present paper aims to explore. In that regard, we would like to draw attention to two results of our discussion:}

Theorem \ref{thm.main} states that every separable subspace of a robust Orlicz space is order isomorphic to a subspace of a classical $L^1$-space over the same underlying measurable space. As a consequence, its elements are dominated by a single probability measure, and the $\CP$-q.s.\ order collapses to an almost sure order -- even for nondominated sets $\CP$ of priors. We thus identify separability to be a property that leads to a collapse of quasi-sure orders, {thereby providing a generalisation of and a different angle for} the result of \cite{BionNadal}. {Hence, even though separable spaces appear naturally in financial applications (see also Section~\ref{discussion spanning} below), it seems questionable in which sense they are actually ``robust''.}

In a similar spirit, our second main result concerns the Dedekind $\sigma$-completeness of sublattices of robust Orlicz spaces, i.e., the existence of least upper bounds for countable bounded subsets. {Dedekind $\sigma$-completeness is necessary for the existence of ``essential suprema", cf.\ \cite{Cohen}, and therefore derives its importance from the widespread use of essential suprema in financial applications, cf.\ \cite{FS16}. In the context of volatility uncertainty, the existence of essential suprema is studied as ``feasibility of aggregations" in \cite{MR2842089} and the literature based thereon.}
In Theorem \ref{thm.super}, we state that typically {\em at most one} separable Dedekind $\sigma$-complete Banach sublattice exists that generates the $\sigma$-algebra, and if it exists, the family of priors $\CP$ is already dominated with uniformly integrable densities. We thereby qualify that what \textit{prevents} nondominated models from being dominated is the lack of \textit{all} order completeness properties for separable Banach sublattices that generate the $\sigma$-algebra.

{One therefore concludes that the robust closure $\CC^p$ of $C_b$ is too similar to the original space $C_b$ both in terms of its order completeness properties and in terms of dominatedness. 

\subsection{Spanning power of options}\label{discussion spanning}
In Section~\ref{sec:options}, we apply our results to the theory of option spanning. These studies date back to \cite{Ross} on finite sets of future states of the economy, and have since been extended to a multitude of model spaces: arbitrary numbers of future states of the economy \cite{GreenJarrow}, the space of continuous functions over a compact Hausdorff space \cite{BrownRoss}, and $L^p$-spaces or even more general ideals of $L^0$ over a probability space \cite{Options,Nachman}. These contributions either do not assume any probabilistic model or consider a fixed prior. It is therefore natural to address option spanning under potentially nondominated uncertainty in the framework of robust Orlicz spaces. 

The option space $\CH_X$ collecting all portfolios of call and put options written on a limited liability claim $X\in L^\Phi(\CP)$ is separable by construction. If $X$ generates the underlying $\sigma$-algebra, one also considers the norm closure $\CC_X:=\cl(\CH_X)$ of all contingent claims which can be approximated by portfolios of call and put options.
We shall take a reverse stance and ask the question: Which consequences can be drawn from a prescription of the \textit{topological spanning power} of $X$, i.e., a statement about the extent of the closure $\CC_X$ in $L^\Phi(\CP)$? In Corollary~\ref{cor:spanningpower}, we prove a converse to the spanning power results on classical $L^p$-spaces.
If a claim has full topological spanning power under uncertainty in that $\CC_X$ is an ideal of $L^\Phi(\CX)$, then the uncertainty is dominated in a very strong sense.}

\subsection{The Fatou property}
In Section~\ref{sec:regular}, we address the (sequential) \textit{Fatou property} from a reverse perspective. The Fatou property is one of the most prominent phenomena in theoretical mathematical finance and plays a crucial role for the theory of risk measures and the Fundamental Theorem of Asset Pricing. In its classical form, it relates sequential order closedness of convex sets of random variables over a probability space $(\Omega,\CF,\P)$ -- a property usually rather straightforward to verify -- to dual representations of these convex sets in terms of measures, see \cite{Maggisetal} for a detailed discussion. 

In contrast to the aforementioned paper, we contemplate structural \textit{consequences} for a generating sublattice $\CH$ of a robust Orlicz space on which this equivalence holds. In fact, we use our theoretical results to argue that studying the sequential Fatou property only makes sense on \textit{ideals} of robust Orlicz spaces. This provides a theoretical justification for the choice of the model space in \cite{Maggisetal}, and we again conclude that in a nondominated framework, separable spaces are usually insufficient for order-related financial applications.

\subsection{Variational preferences}\label{sec:agents}
{Section~\ref{sec:utility} presents economic foundations for the interest in robust Orlicz spaces. More precisely, we demonstrate that the space $L^\Phi(\CP)$ arises naturally in the context of clouds of agents with \textit{variational preferences}. The latter have been axiomatised by~\cite{Variational} and encompass prominent classes of preferences such as multiple prior preferences and multiplier preferences. 
Economic problems like the viability of markets and multi-utility representations of incomplete preferences, however, sometimes require to consider a whole cloud $\mathcal I$ of (representative) agents operating on, say, bounded random variables $X$ and measuring utility with the function
 \[\inf_{\Q\in\CQ_i}\E_\Q[u_i(X)]+c_i(\Q).\]
Here $u_i$ is a scalar utility function, $\CQ_i$ is a set of probability priors, and $c_i$ is a prior-dependent cost function. 
Note that heterogeneity among the agents may require to consider more than one utility function $u_i$. We shall give conditions under which $L^\Phi(\CP)$ is a canonical maximal model space to study all individual preferences simultaneously. Methodologically, this will be achieved by extending each individual preference relation to a \textit{continuous} preference relation on $L^\Phi(\CP)$.}

 \medskip
 
 {\bf Structure of the Paper:} {The paper unfolds as follows. In Section \ref{sec.orlicz}, we study of the top-down approach to robust Orlicz spaces and discuss its basic properties. We derive equivalent conditions for a robust Orlicz space to coincide with a robust multiplicatively penalised $L^1$-space (Theorem \ref{thm:L1}), and show that every separable subspace of $L^\Phi(\CP)$ is order isomorphic to a sublattice of $L^1(\P^*)$ for a suitable probability measure $\P^*$ (Theorem \ref{thm.main}). Section~\ref{sec:alternative} further provides sufficient conditions for the equality $L^\Phi(\CP)=\mf L^\Phi(\CP)$. 
In Section \ref{sec.completion}, we consider sublattices of $L^\Phi(\CP)$. Theorem \ref{thm.super}, Proposition \ref{prop:equivalences0}, and Proposition \ref{prop:equivalences} provide a set of equivalent conditions for the Dedekind $\sigma$-completeness of sublattices of $L^\Phi(\CP)$. In particular, we prove that separability together with Dedekind $\sigma$-completeness for \textit{any} generating sublattice already implies the dominatedness of the set of priors $\CP$. In special yet relevant cases, we give an explicit description of the dual space of sublattices of $L^\Phi(\CP)$ (Proposition~\ref{lem:representation0}). In Section \ref{sec.applications}, we discuss the applications of these theoretical results to robust finance already anticipated in Section~\ref{main results and related}. 
Section \ref{sec.conclusion} concludes. All proofs are relegated to the Appendices~\ref{app.A}--\ref{app.C}.}

\medskip

{\bf Notation:} For a set $S\neq\emptyset$ and a function $f\colon S\to(-\infty,\infty]$, the \emph{effective domain} of $f$ will be denoted by $\dom(f):=\{s\in S\mid f(s)<\infty\}$. For a normed vector space $(\CH,\Norm)$ we denote by $\CH^*$ its dual space and by $\Norm_{\CH^*}$ the operator norm.

Throughout, we consider a measurable space $(\Omega,\CF)$ and a nonempty set $\CP$ of probability measures $\P$ on $(\Omega,\CF)$. 
The latter give rise to an equivalence relation on the real vector space $\CL^0(\Omega,\CF)$ of all real-valued random variables on $(\Omega,\CF)$:  
\[f\sim g\quad :\iff\quad \forall\,\P\in\CP:~\P(f=g)=1.\]
The quotient space $L^0(\CP):=\CL^0(\Omega,\CF)/\sim$ is the space of all real-valued random variables on $(\Omega,\CF)$ up to $\CP$-quasi-sure ($\CP$-q.s.) equality. The elements $f\colon\Omega\to\R$ in the equivalence class $X\in L^0(\CP)$ are called \emph{representatives}, and are denoted by $f\in X$. Conversely, for $f\in\mathcal L^0(\Omega,\CF)$, $[f]$ denotes the equivalence class in $L^0(\CP)$ generated by $f$. 
 For $X$ and $Y$ in $L^0(\CP)$,
$$X\peq Y\quad:\iff\quad\forall\, f\in X\,\forall \,g\in Y\,\forall\,\P\in\CP:~\P(f\leq g)=1,$$
defines a vector space order on $L^0(\CP)$, the \emph{$\CP$-q.s.\ order} on $L^0(\CP)$, and $(L^0(\CP),\peq)$ is a vector lattice. In fact, for $X,Y\in L^0(\CP)$ and representatives $f\in X, g\in Y$, the formulae
\[X\wedge Y=[f\wedge g]\quad\text{and}\quad X\vee Y=[f\vee g]\]
hold for the minimum and the maximum, respectively. We denote the vector sublattice of all bounded real-valued random variables up to $\CP$-q.s.\ equality by $\Linfty$. The latter is a Banach lattice, when endowed with the norm
\[\|X\|_{L^\infty(\CP)}:=\inf\big\{m>0\, \big|\, X\peq m\ind_\Omega\big\},\quad X\in\Linfty.\]
As usual, $\ca$ denotes the space of all signed measures on $(\Omega,\mathcal F)$ with finite total variation. We denote by $\ca_+$ and $\ca_+^1$ the subset of all finite measures and probability measures, respectively. For $\mu\in \ca$, let $|\mu|$ denote the total variation measure of $\mu$. {We write
\begin{itemize}
\item $\CQ\ll\mf R$ for $\emptyset\neq\CQ$ and $\mf R\subset \ca$ if, for $N\in\CF$, $\sup_{\nu\in \mf R}|\nu|(N)=0$ implies $\sup_{\mu\in \CQ}|\mu|(N)=0$.
\item $\CQ \approx \mf R$ if $\CQ \ll \mf R$ and $\mf R \ll \CQ$. 
\item $\mu \ll \mf R$, $\mf R\ll \mu$, and $\mf R \approx \mu$ if the set $\CQ$ above consists of a sigle measure $\mu\in \ca$, i.e., $\CQ=\{\mu\}$.
\item $\ca(\CP):=\{\mu\in \ca\, |\, \mu\ll \CP\}$ for the space of all countably additive signed measures absolutely continuous w.r.t.\ $\CP$.
\end{itemize}}
The subsets $\ca_+(\CP)$ and $\ca_+^1(\CP)$ are defined  analogously. 
{Note that both $\ca$ and $\ca(\CP)$ are lattices w.r.t.\ the setwise order $\peq_\CF$ which sets $\mu\peq_\CF\nu$ iff $\mu(A)\le \nu(A)$ holds for all $A\in\CF$.
For all $\mu\in \ca_+(\CP)$, $X\in L^0(\CP)_+$, and $f,g\in X$, $\int f\,{\rm d}\mu$ and $\int g\,{\rm d}\mu$ are well-defined and satisfy
\[\int f\,{\rm d}\mu=\int g\,{\rm d}\mu.\]
If $X\in L^0(\CP)$, $f\in X$ arbitrary, $f^\pm$ and $\mu^\pm$ denote the positive and negative part of $f$ and $\mu$, respectively, and if $\int f^+\,{\rm d}\mu^++\int f^-{\rm d}\mu^-$ or $\int f^+\,{\rm d}\mu^-+\int f^-\,{\rm d}\mu^+$ is a real number, then we shall write }
\[\int X\,{\rm d}\mu:=\int f^+\,{\rm d}\mu^++\int f^-{\rm d}\mu^--\int f^+\,{\rm d}\mu^--\int f^-\,{\rm d}\mu^+.\]

\section{Robust Orlicz spaces: definition and first properties}\label{sec.orlicz}

In this section, we introduce robust versions of Orlicz spaces, the main object of interest of this manuscript, and investigate their basic properties. For the theory of classical Orlicz spaces, we refer to \cite[Chapter 2]{Edgar}. An \emph{Orlicz function} is a function $\phi\colon[0,\infty)\to[0,\infty]$ with the following three properties:
\begin{enumerate}[(i)]
\item $\phi$ is lower semicontinuous, nondecreasing, and convex.
\item $\phi(0)=0$.
\item There are $x_0,x_1>0$ with $\phi(x_0)\in [0,\infty)$ and $\phi(x_1)\in (0,\infty]$.\footnote{~This definition precludes triviality of $\phi$, i.e.\ the cases $\phi\equiv 0$ and $\phi=\infty \cdot \ind_{(0,\infty)}$.} 
\end{enumerate}
Throughout this section, we consider a general measurable space $(\Om,\CF)$, a nonempty set of probability measures $\CP$, a family $\Phi=(\phi_\P)_{\P\in\CP}$ of Orlicz functions, and define
\[
\phi_\Max(x):=\sup_{\P\in \CP} \phi_\P(x),\quad \text{for all }x\ge 0.
\]
By definition, $\phi_\Max\colon [0,\infty)\to [0,\infty]$ is a lower-semicontinuous, nondecreasing, and convex function with $\phi_\Max(0)=0$. However, in general, $\phi_\Max$ is not an Orlicz function, since $\phi_\Max(x_0)\in [0,\infty)$ for some $x_0\in (0,\infty)$ cannot be guaranteed.

\subsection{Robust Orlicz spaces and penalised versions of robust $L^p$-spaces}

\begin{definition}\label{def:robustOrlicz}\
 For $X\in L^0(\CP)$, the \emph{($\Phi$-)Luxemburg norm} is defined via
 \begin{equation}\label{eq.luxemburg}
  \|X\|_{L^\Phi(\CP)}:=\inf\big\{ \la>0\, \big|\,\sup_{\P\in \CP} \E_\P\left[\phi_\P(\lambda^{-1}|X|)\right]\leq 1\big\}\in [0,\infty].
 \end{equation}
The \emph{($\Phi$-)robust Orlicz space} is defined by $L^\Phi(\CP):=\dom(\Norm_{L^\Phi(\CP)})$.
\end{definition}
Note that, by definition, for $X\in L^0(\CP)$,
\[
 \|X\|_{L^\Phi(\CP)}=\sup_{\P\in \CP}\|X\|_{L^{\phi_\P}(\P)},
\]
where $\Norm_{L^{\phi_\P}(\P)}$ is given by \eqref{eq.luxemburg} with $\CP=\{\P\}$, for all $\P\in \CP$.

\begin{example}\label{ex.1}
 Let $(\Omega,\CF)$ be a measurable space, $\CP$ a nonempty set of probability priors, and $\phi\colon [0,\infty)\to [0,\infty]$ be an Orlicz function.
 \begin{enumerate}
  \item[(1)] For an arbitrary function $\ga\colon \CP\to [0,\infty)$, consider
  \[
   \phi_\P(x):=\frac{\phi(x)}{1+\ga(\P)},\quad \text{for }x\geq0.
  \]
  This leads to an \emph{additively penalised robust Orlicz space} with Luxemburg norm
  \[
   \|X\|_{L^{\Phi}(\CP)}=\inf\big\{ \la>0\, \big|\,\sup_{\P\in \CP} \E_\P\left[\phi(\lambda^{-1}|X|)\right]-\ga(\P) \leq 1\big\},\quad \text{for }X\in L^0(\CP).
  \]
  For $\phi:=\infty\cdot \ind_{(1,\infty)}$, {we observe $\phi_\P=(1+\ga(\P))^{-1}\phi=\phi$, for all $\P\in\CP$.} Hence, the Luxemburg norm is in that case independent of $\ga$ and given by
  \[
   \|X\|_{L^\Phi(\CP)}=\sup_{\P\in \CP}\|X\|_{L^\infty(\P)}=\|X\|_{L^\infty(\CP)},\quad \text{for }X\in L^0(\CP).
  \]
  Introducing the, up to a sign, convex monetary risk measure
  \[
   \rho(X):=\sup_{\P\in \CP} \E_\P[X]-\ga(\P)\in [0,\infty],\quad \text{for }X\in L^0(\CP)_+,
  \]
  the robust Luxemburg norm can be expressed as
  \[
   \|X\|_{L^\Phi(\CP)}=\inf\left\{ \la>0\, \left|\,\rho\left(\phi(\lambda^{-1}|X|)\right)\leq 1\right.\right\},\quad \text{for }X\in L^0(\CP).
  \]
  \item[(2)] For $\theta\colon \CP\to (0,\infty)$ with $\sup_{\P\in \CP}\theta(\P)<\infty$, we consider
  \[
   \phi_\P(x):=\phi\big(\theta(\P)x\big),\quad\text{for }\P\in \CP\text{ and } x\geq 0.
  \]
  This leads to a \emph{multiplicatively penalised robust Orlicz space} with Luxemburg norm
  \[
   \|X\|_{L^\Phi(\CP)}=\sup_{\P\in \CP}\theta(\P)\|X\|_{L^\phi(\P)},\quad \text{for }X\in L^0(\CP).
  \]
  For $p\in [1,\infty)$ and $\phi(x)=x^p$, $x\ge 0$, we obtain the weighted robust $L^p$-norm
  \[
   \|X\|_{L^\phi(\CP)}=\sup_{\P\in \CP}\theta(\P)\|X\|_{L^p(\P)},\quad \text{for }X\in L^0(\CP),
  \]
  and, for $\phi(x)=\infty\cdot \ind_{(1,\infty)}$, the Luxemburg norm is given by
  \[
   \|X\|_{L^\Phi(\CP)}=\sup_{\P\in \CP}\theta(\P)\|X\|_{L^\infty(\P)},\quad \text{for }X\in L^0(\CP).
  \]
  The resulting spaces will be referred to as \emph{weighted robust $L^p$-spaces}, for $1\leq p\leq \infty$.
 \end{enumerate}
\end{example}

{The next proposition records that, as in the classical case, robust Orlicz spaces are Banach lattices. There and in the following we denote by $\ca\big(L^\Phi(\CP)\big)$} the set of all signed measures $\mu\in \ca(\CP)$ for which each $X\in L^\Phi(\CP)$ is $|\mu|$-integrable and the map
\begin{equation}\label{eq.measfunct}
 L^\Phi(\CP)\to \R, \quad X\mapsto \int X\,{\rm d}|\mu|
\end{equation}
is continuous. Moreover, we set 
\begin{center}$\ca_+\big(L^\Phi(\CP)\big):=\ca\big(L^\Phi(\CP)\big)\cap \ca_+$\quad and\quad$\ca_+^1\big(L^\Phi(\CP)\big):=\ca\big(L^\Phi(\CP)\big)\cap \ca_+^1$.\end{center}

\begin{proposition}\label{prop:Banach}
The following assertions hold:
\begin{itemize}
\item[\tn{(1)}]The space $\big(L^\Phi(\CP),\peq, \Norm_{L^\Phi(\CP)}\big)$ is a Dedekind $\sigma$-complete Banach lattice.
\item[\tn{(2)}]$L^\Phi(\CP)\subset L^0(\CP)$ is an ideal. 
\item[\tn{(3)}]For all $\P\in \CP$, $a_\P>0$, and $b_\P\geq 0$ with $a_\P x-b_\P\le \phi_\P(x)$ for all $x\ge 0$, 
 \begin{equation}\label{eq:normP}\E_\P [|X|]\leq \frac{1+b_\P}{a_\P}\|X\|_{L^\Phi(\CP)},\quad \text{for }X\in L^\Phi(\CP).\end{equation}
\item[\tn{(4)}]$\CP\subset \ca_+^1\big(L^\Phi(\CP)\big)$.
\end{itemize}
\end{proposition}
\begin{remark}\label{rem:measures}
The lattice norm property of $\Norm_{L^\Phi(\CP)}$ leads to two conclusions: (i) For each $\mu\in\ca(L^\Phi(\CP))$, the functional $L^\Phi(\CP)\ni X\mapsto\int X\,{\rm d}\mu$ is continuous. This is due to the fact that the Radon-Nikodym derivative $\frac{{\rm d}\mu}{{\rm d}|\mu|}$ takes values in $[-1,1]$ $|\mu|$-almost everywhere. (ii) $\ca(L^\Phi(\CP))$ is a vector sublattice of $\ca(\CP)$. 
\end{remark}

\begin{example}
 Suppose $\CH\subset L^0(\CP)$ is an ideal which is a Banach lattice when endowed with a norm $\Norm_\CH$. Furthermore assume the norm is completely determined by $\si$-finite measures, i.e., there is a set $\mf D\ll\CP$ of $\sigma$-finite measures such that, for all $X\in \CH$,
 \[
  \|X\|_\CH=\sup_{\mu\in \mf D} \int |X|\,{\rm d}\mu.
 \]
 Then $\CH$ is a robust Orlicz space after a potential modification of $\CP$.
\end{example}

{The following theorem proves that a robust Orlicz space reduces to a weighted robust $L^1$-space if and only if it contains all bounded random variables.}

\begin{theorem}\label{thm:L1}
 The following statements are equivalent:
 \begin{enumerate}[(\tn{1})]
  \item[\tn{(1)}] $\Linfty\subset L^\Phi(\CP)$.
  \item[\tn{(2)}] $\phi_\Max$ is an Orlicz function, i.e., there exists some $x_0\in (0,\infty)$ with $\phi_\Max(x_0)\in [0,\infty)$.
  \item[\tn{(3)}] There exists a nonempty set of probability measures $\CQ\subset \ca_+^1\big(L^\Phi(\CP)\big)$ satisfying $\CP\subset\CQ$ and a bounded weight function $\theta\colon \CQ\to (0,\infty)$ 
  such that 
  \begin{center}$\Norm_{L^\Phi(\CP)}=\sup_{\Q\in \CQ}\theta(\Q)\Norm_{L^1(\Q)}$.\end{center}
That is, $L^\Phi(\CP)$ is a weighted robust $L^1$-space.
\item[\tn{(4)}] There is a constant $\kappa>0$ such that 
 \[
  \|X\|_{L^\Phi(\CP)}\leq \kappa\|X\|_{L^\infty(\CP)},\quad\text{for } X\in \Linfty.
 \]
  \end{enumerate}
\end{theorem}

\smallskip

\begin{remark}\label{rem:RoyChakraborty}
Theorem~\ref{thm:L1} clarifies the relation of our work to \cite{MR856304,MR1038172}. Leaving technical subtleties aside, two differences are apparent at first sight. Firstly, \cite{MR856304,MR1038172} consider Orlicz spaces of vector-valued functions; we consider scalar-valued functions. Secondly, our concept appears to be more general in the scalar case as it involves a family $\Phi=(\phi_\P)_{\P\in\CP}$ of prior-dependent Orlicz functions instead of a single Orlicz function as in \cite{MR856304,MR1038172}. As weighted robust $L^1$-spaces are definitely covered by \cite{MR856304,MR1038172}, Theorem~\ref{thm:L1}, however, shows that $\Linfty\subset L^\Phi(\CP)$ is necessary and sufficient for our framework to embed in theirs. Note that we impose the condition $\Linfty\subset L^\Phi(\CP)$ throughout Section~\ref{sec.completion}, cf.\ Assumption~\ref{assumption}. 

The identification of Theorem~\ref{thm:L1}(3) typically requires changing to a new set $\CQ$ of priors though. In view of this obstacle and the economic motivation for robust Orlicz spaces outlined in Sections~\ref{sec:agents} and~\ref{sec:utility}, we therefore believe the consideration of a family $\Phi=(\phi_\P)_{\P\in\CP}$ is more natural.
\end{remark}

For $\Q\in \ca_+^1(L^\Phi(\CP))$, we define the \emph{canonical projection} $J_\Q\colon L^\Phi(\CP)\to L^1(\Q)$ via
\[
 J_\Q(X):=\big\{g\in \CL^0(\Omega,\CF)\,\big|\, \exists\,f\in X\colon \Q(f\neq g)=0\big\},\quad \text{for }X\in L^\Phi(\CP).
\]
Since $\Q\in \ca_+^1\big(L^\Phi(\CP)\big)$, $J_\Q$ is well defined, linear, continuous, and a lattice homomorphism, i.e., it preserves the order in that
\begin{center}$J_\Q(X\wedge Y)=J_\Q(X)\wedge J_\Q(Y)$,\quad for $X,Y\in L^\Phi(\CP)$.\end{center}
However, in general, it fails to be a lattice isomorphism onto its image, i.e., it is not injective. {Still, the following result holds: On separable subspaces of $L^\Phi(\CP)$, the $\CP$-q.s.~order collapses to a $\P^*$-a.s.~order for some $\P^*\in \ca_+^1(L^\Phi(\CP))$.}

\begin{theorem}\label{thm.main}
Suppose $\CH$ is a separable subspace of $L^\Phi(\CP)$. Then, there is a probability measure $\P^*\in \ca_+^1\big(L^\Phi(\CP)\big)$ such that $\CH$ is isomorphic to a subspace of $L^1(\P^*)$ via the canonical projection $J_{\P^*}$.\ In particular, the following assertions hold:
\begin{itemize}
\item[(1)]$\P^*$ defines a strictly positive linear functional on $\CH$.
\item[(2)]The $\CP$-q.s.\ order and the $\P^*$-a.s.\ order coincide on $\CH$. 
\item[(3)]If $\CP$ is countably convex, $\P^*$ can be chosen as an element of $\CP$.
\end{itemize}
\end{theorem}

\begin{corollary}\label{cor.main}
 Assume that one of the equivalent conditions of Theorem \ref{thm:L1} is satisfied. Let $\CQ\subset \ca_+^1\big(L^\Phi(\CP)\big)$ be as described in point \tn{(3)}.
 Then, for every separable subspace $\CH$ of $L^\Phi(\CP)$, there exists a countable set $\CQ_\CH\subset \CQ$ such that
 $$\|X\|_{L^\Phi(\CP)}=\sup_{\Q\in \CQ_\CH}\theta(\Q)\|X\|_{L^1(\Q)},\quad\text{for all }X\in \CH.$$
\end{corollary}

Theorem \ref{thm.main} is akin to results of Nagel, see \cite[Theorem 2.7.8]{MeyNie}. However, these use Kakutani representation and isomorphisms between a multitude of Banach lattices, while our approach does not require a change of the underlying measurable space or topological structure. Similar representation results for general Banach lattices can also be found in \cite[Theorem 1.b.14]{MR540367} and the references provided therein.

\begin{example}\label{ex.doubpen}\
 \begin{enumerate}[(1)]
  \item We consider the setup of Example \ref{ex.1}. Let $\theta\colon\CP\to(0,\infty)$ with $c:=\sup_{\P\in \CP}\theta(\P)<\infty$, $\gamma\colon\CP\to[0,\infty)$, and $\phi$ be a joint Orlicz function. Let
  \[
   \phi_\P(x):=\frac{\phi\big(\theta(\P)x\big)}{1+\gamma(\P)}, \quad \text{for }x\geq 0,
  \]
  corresponding to the case of a \emph{doubly penalised robust Orlicz space}. Then, for $x_0\in (0,\infty)$ with $c x_0\in\dom(\phi)$,
  \[\phi_\Max (x_0)=\sup_{\P\in\CP}\phi_\P(x_0)=\sup_{\P\in\CP}\frac{\phi\big(\theta(\P)x_0\big)}{1+\gamma(\P)}\le \phi(c x_0)<\infty.\]
  By Proposition \ref{thm:L1}, we obtain that $L^\Phi(\CP)$ is a weighted robust $L^1$-space.
  \item {Although this result could, of course, also be obtained in a more direct manner, Theorem~\ref{thm:L1} shows that the classical space $L^\infty(\P^*)$ over a probability space $(\Omega,\CF,\P^*)$ is a robust $L^1$-space.} Indeed, let $\CP$ be the set of all probability  measures $\P$ on $(\Om,\CF)$ that are absolutely continuous with respect to $\P^*$. Consider $\phi_\P(x)=x$ for all $x\geq 0$ and $\P\in \CP$, leading a robust $L^1$-space over $\CP$. Then,
  \[
   \|X\|_{L^\Phi(\CP)}=\|X\|_{L^\infty(\P^*)},\quad \text{for }X\in L^0(\CP)=L^0(\P^*).
  \]
  \item Let $\P^*$ be a probability measure on $(\Om,\CF)$, and consider a convex monetary risk measure $\rho\colon L^\infty(\P^*)\to\R$, which enjoys the Fatou property and satisfies $\rho(0)=0$. The dual representation, up to a sign,
  \[\rho(X)=\sup_{Z\in\dom(\rho^*)\cap L^1(\P^*)}\E[ZX]-\rho^*(Z),\quad \text{for }X\in L^\infty(\P^*),\]
  is a well-known consequence, where $\rho^*$ is the convex conjugate of $\rho$.
  In the situation of Example \ref{ex.1}(1), set
  \begin{align*}\CP&:=\big\{Z\,{\rm d}\P^*\, \big|\,  Z\in\dom(\rho^*)\cap L^1(\P^*)\big\},\\
   \gamma\big(Z\,{\rm d}\P^*\big)&:=\rho^*(Z),\quad \text{for }Z\in \dom(\rho^*)\cap L^1(\P^*),\\
   \phi_\P(x)&:=x,\quad \text{for }x\geq 0\text{ and }\P\in\CP.
  \end{align*}
  Then, $L^\Phi(\CP)$ contains $\Linfty$ as a sublattice. In general, we have $\CP\ll\P^*$, but $\CP\approx\P^*$ may fail without further conditions on $\rho$. We can always define the ``projection'' 
  \[
   \w\rho(Y):=\rho\big(J^{-1}(Y)\big), \quad\text{for }Y\in\Linfty,
  \]
  though, where $J\colon L^\infty(\P^*)\to\Linfty$ is the natural projection. In that case, $L^\Phi(\CP)$ serves as the maximal sensible domain of definition of $\w\rho$. Various aspects of such spaces have been studied in~\cite{KupSvi,Liebrich,Owari,Pichler}. 
 \end{enumerate}
\end{example}

\subsection{An alternative path to robust Orlicz spaces}\label{sec:alternative}

{In this subsection, we focus on a concept of robust Orlicz spaces which does not require the worst-case ansatz in the definition of a robust Luxemburg norm over all models $\P\in\CP$, cf.\ \eqref{eq.luxemburg}. We have already sketched in Section~\ref{sec:agents} the economic application of studying (variational) preferences of a cloud of agents simultaneously. This point of view does not seem to require the modelling assumption of the worst case {\em a priori} to produce the largest commodity space on which the analytic behaviour of each preference relation involved can be captured well. An alternative would be provided by the space
 \begin{align}\begin{split}\label{eq:curlyL}\mf L^\Phi(\CP)&:=\big\{X\in L^0(\CP)\,\big|\, \forall\,\P\in\CP\,\exists\,\alpha>0:~\E_\P[\phi_\P(\alpha|X|)]<\infty\big\}\\
 &=\big\{X\in L^0(\CP)\,\big|\, \forall\,\P\in\CP:~\|X\|_{L^{\phi_\P}(\P)}<\infty\big\}.\end{split}\end{align}
A special case of this space has, e.g., been studied in \cite{Nutz,MR2842089}. The space $\mf L^\Phi(\CP)$ collects minimal agreement among all agents under consideration, that is, they all can attach a well-defined utility to each of the objects in $\mf L^\Phi(\CP)$.} One can show that $\mf L^\Phi(\CP)$ is a vector sublattice of $L^0(\CP)$. Moreover, the inclusion $L^\Phi(\CP)\subset\mf L^\Phi(\CP)$ holds independently of $\Phi$ and can be strict, as the following example demonstrates.

\begin{example}\label{ex.eqorlicz}\
  Fix two constants $0<c<1<C$ and consider the case where $\Om=\R$ is endowed with the Borel $\sigma$-algebra $\CF$, and $\CP$ is given by the set of all probability measures $\P$ which are equivalent to $\P^*:=\CN(0,1)$ with bounded density $c\le\frac{{\rm d}\P}{{\rm d}\P^*}\le C$. Moreover, fix a partition $(\CP_n)_{n\in\N}$ of $\CP$ into nonempty sets. We set
  \[
   \phi_\P(x):=x^n,\quad \text{for } x\geq 0,\, n\in \N,\text{ and }\P\in \CP_n. 
  \]
  Then,
  \[
  \mf L^\Phi(\CP)=\big\{X\in L^0(\P^*)\, \big|\, \forall\,n\in \N:~\E_{\P^*}[|X|^n]<\infty\big\},
  \]
  and thus $U\in\mf L^\Phi(\CP)$ if $U\colon \Om\to \R$ is the identity, i.e., $U\sim\CN(0,1)$ under $\P^*$. However, Stirling's formula implies that, for all $\al>0$,
  \[\sup_{\P\in\CP}\E_\P [\phi_\P(\al|U|)]\ge c\sup_{n\in\N}\E_{\P^*}[\al^n|U|^n]=\infty,\]
  and $U\notin L^\Phi(\CP)$ follows. It is easy to see that $\mf L^\Phi(\CP)$ is a Fr\'echet space, but not a Banach space. 
\end{example}

The next proposition shows that $L^\Phi(\CP)$ can always be seen as a space of type \eqref{eq:curlyL} if $\Linfty\subset L^\Phi(\CP)$, and more can be said if $\CP$ is countably convex. Note that the equivalence of (1) and (2) in Proposition~\ref{prop:equality} would also follow from Theorem~\ref{thm:L1} and \cite[Theorem 4.4]{MR1038172}. We shall give a self-contained proof below.

\begin{proposition}\label{prop:equality}
The following statements are equivalent: 
\begin{itemize}
\item[(1)]$\Linfty\subset L^\Phi(\CP)$.
\item[(2)]There is a set of probability measures $\mf R\subset\ca^1_+(L^\Phi(\CP))$ and a family $\Psi=(\psi_\Q)_{\Q\in\mf R}$ of Orlicz functions such that $\mf R\approx\CP$ and  
 \[L^\Phi(\CP)=\mf L^\Psi(\mf R).\]
 \end{itemize}
 In particular, if $\CP$ is countably convex and there exist constants $(c_\P)_{\P\in \CP}\subset (0,\infty)$ such that
 \begin{equation}\label{eq.thm:equality}
  \phi_\Max(x)\leq \phi_\P(c_\P x),\quad \text{for all }x\geq 0 \text{ and }\P\in \CP,
 \end{equation}
 then \tn{(1)} and \tn{(2)} hold and one can choose $\mf R=\CP$ as well as $\Psi=\Phi$ or $\Psi=(\phi_\Max)_{\P\in\CP}$. 
\end{proposition}

While the assumption of countable convexity of the set $\CP$ is typically satisfied in the context of coherent risk measures, it fails for penalisations arising from convex risk measures, i.e., doubly penalised Orlicz spaces as introduced in Example~\ref{ex.doubpen}(1). The following theorem shows that, in this particular case, the assumption of countable convexity of the set $\CP$ can be further relaxed.

\begin{theorem}\label{thm:equality}
 Suppose that $\CP$ is convex. Furthermore, assume that $\Phi$ is doubly penalised with joint Orlicz function $\phi$, multiplicative penalisation $\theta$, and convex additive penalty function $\gamma\colon\CP\to[0,\infty)$ with countably convex lower level sets.
Then,
\[\mf L^\Phi(\CP)=L^\Phi(\CP).\]
\end{theorem}

\begin{example}
{For an additive penalty function as demanded in Theorem~\ref{thm:equality}, consider a convex monetary risk measure $\rho:\mathcal L^\infty(\Omega,\CF)\to\R$ satisfying $\rho(0)=0$ and $\CP:=\dom(\rho^*)\cap \ca_+^1\neq\emptyset$. Then the additive penalty $\gamma(\P):=\rho(\P)$, $\P\in\CP$, is convex and has countably convex lower level sets. However, the set $\CP$ in total} is typically not countably convex, as the choice $\Omega=\N$, $\CF=2^\N$, and 
\[\rho(f)=\sup_{n\in\N}f(n)-2^{2n},\quad\text{for }f\in\mathcal L^\infty(\Omega,\CF),\]
demonstrates: The Dirac measure $\delta_n$ lies in $\dom(\rho^*)\cap \ca^1_+$, $n\in\N$, but $\rho^*(\sum_{n=1}^\infty2^{-n}\delta_n)=\infty$. 
\end{example}
 
\begin{remark}
Assume that, in the situation of Theorem~\ref{thm:equality}, the multiplicative penalty is $\theta\equiv 1$. Then, there are two equally consistent ways to translate convergence in $L^\phi(\P)$ to a robust setting given by the set $\CP$ of priors. One could either declare a net $(X_\alpha)_{\al\in I}$ to be convergent if it (i) converges with respect to each seminorm $\Norm_{L^\phi(\P)}$, for $\P\in\CP$, at equal or comparable speed to the same limit, or (ii) converges to the same limit with respect to each seminorm $\Norm_{L^\phi(\P)}$, for $\P\in\CP$. Convergence (i) is reflected by the norm $\Norm_{L^\Phi(\CP)}$, and the equality of speeds may be relaxed by the additive penalty. Convergence (ii) would lead to the natural choice of a topology on $\mf L^\Phi(\CP)$. Even though $\mf L^\Phi(\CP)=L^\Phi(\CP)$ holds, convergence (ii) might not be normable or even sequential. However, having both options at hand provides a degree of freedom to reflect different economic phenomena on an applied level.
\end{remark}

\section{Generating sublattices of robust Orlicz spaces}\label{sec.completion}

{Recall that a vector lattice $(\CX,\peq)$ is 
\begin{itemize} 
\item Dedekind $\sigma$-complete if each countable subset $\mathcal C\subset\CX$ possessing an order upper bound $y\in\CX$ has a least upper bound (denoted by $\sup\CC$).
\item Dedekind complete if $\sup\CC$ exists for all subsets $\mathcal C\subset\CX$ possessing an order upper bound $y\in\CX$.
\item super Dedekind complete if it is Dedekind complete and suprema are attained by countable subsets. 
\end{itemize}
$\Phi$-robust Orlicz spaces not only have the desirable Banach space property, but also behave reasonably well as vector lattices. Indeed, they are Dedekind $\sigma$-complete ideals in $L^0(\CP)$ with respect to the $\CP$-q.s.\ order. Moreover, using arguments as in \cite[Lemma 8]{surplus}, (super) Dedekind completeness of $L^0(\CP)$ implies (super) Dedekind completeness of $L^\Phi(\CP)$, and the converse implications hold in the situation of Theorem~\ref{thm:L1}.}

In contrast to the top-down construction of $\Phi$-robust Orlicz spaces, one {may contemplate building} such a space bottom-up, a path taken in, e.g., \cite{Beissner,BionNadal,Peng}. Starting with a space of \textit{test random variables}, one may close this test space in the larger ambient space $L^\Phi(\CP)$ with respect to the risk-uncertainty structure as given by $\Norm_{L^\Phi(\CP)}$. Such a procedure leads to proper subspaces in general. The existing literature typically discusses (special cases of) these spaces as Banach spaces without going into further detail on their order-theoretic properties. {The objective of this section is to fill this gap.}

\begin{assumption}\label{assumption}
Throughout this section, we assume that there exists some $x_0\in (0,\infty)$ with $\phi_\Max(x_0)\in [0,\infty)$, or, equivalently, that $\Linfty\subset L^\Phi(\CP)$.
\end{assumption}

In the following, we consider a sublattice $\CH$ of $L^\Phi(\CP)$ containing the equivalence class of the constant function $\ind_\Omega$. We also assume that $\CH$ \textit{generates} $\CF$ in that the $\sigma$-algebra $\sigma(\CL)$ generated by the lattice $\CL:=\{f\in\CL^0(\Omega,\CF)\mid [f]\in\CH\}$ equals $\CF$. 
Note that the latter assumption does not restrict generality and merely simplifies the exposition of our results. They transfer to smaller $\sigma$-algebras otherwise. 
By $\CC$ we denote the $\Norm_{L^\Phi(\CP)}$-closure of $\CH$ in $L^\Phi(\CP)$, i.e.
$$\CC={\rm cl}(\CH).$$
We define the subspaces $\ca(\CH)$ and $\ca(\CC)$ of $\ca(\CP)$ in complete analogy with $\ca(L^\Phi(\CP))$ (cf.\ equation \eqref{eq.measfunct}). Using Remark~\ref{rem:measures}, one can show that, for each $\mu\in\ca(\CH)$, the functional $\CH\ni X\mapsto\int X\,{\rm d}\mu$ is continuous.

\begin{definition}\label{def:sigmaorder}
A possibly nonlinear functional $\ell\colon \CX\to\R$ on a vector lattice $(\CX,\peq)$ is \emph{$\sigma$-order continuous} if it has the following two properties:
\begin{enumerate}[(i)]
\item for all $x,y\in\CX$, the set $\{\ell(z)\mid x\peq z\peq y\}\subset\R$ is bounded.
\item {$\lim_{n\to\infty}|\ell(x_n)|=0$ holds for all sequences $(x_n)_{n\in\N}\subset\CX$ for which there is another sequence $(y_n)_{n\in\N}\subset\CX_+$ satisfying $|x_n|\peq y_n$ for all $n\in\N$, and $y_n\downarrow 0$ (i.e., $y_{n+1}\peq y_n$ for all $n\in \N$, and $\inf_{n\in\N}y_n=0$ in $\CX$).}
\end{enumerate}
\end{definition}

In a first step, we characterise $\sigma$-order continuous linear functionals on $\CC$ and $\CH$. 

\begin{lemma}\label{lem:sigmaorder}
For each $\sigma$-order continuous linear functional $\ell\colon\CH\to\R$ there is a unique signed measure $\mu\in\ca(\CP)$ such that, for all $X\in\CH$, all representatives of $X$ are $|\mu|$-integrable and
\[\ell(X)=\int X{\rm d}\mu.\]
In particular, $\ell$ satisfies $\ell(X)\ge 0$ for all {$X\in\CH$ with $X\succeq 0$ iff }the associated measure $\mu$ lies in $\ca_+(\CP)$. 
\end{lemma}

{We hence identify the space of all $\sigma$-order continuous linear functionals on $\CH$ with a subset $\ca^\sigma(\CH)\subset\ca(\CP)$ and positive $\sigma$-order continuous linear functionals with the set $\ca_+^\sigma(\CH):=\ca^\sigma(\CH)\cap \ca_+$.}

\begin{lemma}\label{lem:representation00}
 Assume that $\Norm_{L^\Phi(\CP)}$ is $\sigma$-order continuous on $\CH$. Then, $$\CH^*=\ca(\CH)=\ca^\sigma(\CH)\cap \CH^*.$$ 
\end{lemma}

\begin{proposition}\label{lem:representation0}
The space $\big(\CC,\peq,\Norm_{L^\Phi(\CP)}\big)$ is a Banach lattice and $\ca(\CC)=\ca(\CH)$. If $\CH\subset L^\infty(\CP)$ and $\Norm_{L^\Phi(\CP)}$ is $\sigma$-order continuous on $\CH$, then 
$$\CC^*=\ca(\CC)=\ca(\CH)=\ca^\sigma(\CH)\cap \CC^*.$$
\end{proposition}

Throughout the remainder of this section, the closure $\cl\big(L^\infty(\CP)\big)$ of $\Linfty$ plays a fundamental role. The following lemma is a slight generalisation of \cite[Proposition 18]{Peng} and provides an explicit description of the closure of $\Linfty$ in our setup.

\begin{lemma}\label{lem.Linfty}
For $X\in L^\Phi(\CP)$, the following statements are equivalent:
\begin{enumerate}
\item[\tn{(1)}]$X\in \cl\big(L^\infty(\CP)\big)$.
\item[\tn{(2)}] For all $\al>0$, $\lim_{n\to \infty}\sup_{\P\in \CP}\E_\P\big[\phi_\P(\al |X|)\mbf 1_{\{|X|> n\}}\big]= 0.$
\item[\tn{(3)}]
$\lim_{n\to \infty}\big\|X\ind_{\{|X|> n\}}\big\|_{L^\Phi(\CP)}=0$.
\end{enumerate}
\end{lemma}

For the remaining results of this section, we emphasise that, if we view $\CH$ or $\CC$ as spaces of \textit{measurable functions}, two properties should not be far fetched:
\begin{enumerate}
 \item[(i)] Dedekind $\sigma$-completeness,
 \item[(ii)] many positive functionals which are integrals with respect to a measure are $\sigma$-order continuous.\footnote{~Tellingly, the early literature on vector lattices refers to $\sigma$-order continuous linear functionals as ``integrals".}
\end{enumerate}
The following theorem shows that, if $\CP$ is nondominated, the Banach lattice $\CC$ cannot be  separable \textit{and} simultaneously have the mild order completeness property of Dedekind $\sigma$-completeness. The proof adopts general ideas from the theory of Banach lattices, see, for example, \cite[Theorem 1.b.14]{MR540367}.

\begin{theorem}\label{thm.super}
Suppose that the Banach lattice $\big(\CC,\peq,\Norm_{L^\Phi(\CP)}\big)$ is separable, and let $\P^*$ be a probability measure as in Theorem~\ref{thm.main}. Then, the following statements are equivalent:
\begin{itemize}
  \item[(1)] $\CC$ is Dedekind $\sigma$-complete.
  \item[(2)] $\CC$ is super Dedekind complete.
  \item[(3)] $\CC=\cl\big(L^\infty(\CP)\big)$. 
  \item[(4)] $\CC$ is an ideal in $L^\Phi(\CP)$.
   \item[(5)]$\CC^*=\ca(\CC)=\ca(L^\Phi(\CP))\approx\P^*$ and the unit ball therein is weakly compact in $L^1(\P^*)$.
\end{itemize}
Moreover, they imply both of the following assertions:
\begin{itemize}
  \item[(6)]$\CP\approx \P^*$.
  \item[(7)]If, additionally, 
\begin{equation}\label{eq:growth}\inf_{\P\in \CP}\phi_\P(x_0)\in (0,\infty]\quad\text{for some }x_0\in (0,\infty),\end{equation}
the set $\big\{\frac{{\rm d}\P}{{\rm d}\P^*}\, \big|\, \P\in \CP\big\}$ of densities of priors in $\CP$ is  uniformly $\P^*$-integrable.
 \end{itemize}
 \end{theorem}

We thus see that, in typical situations encountered in the literature, all order completeness properties agree, and their validity usually implies dominatedness of the underlying set of priors in a particularly strong form.
Although separability is a desirable property from an analytic point of view, we have hereby shown that it has very strong implications for robust spaces. One may wonder what happens if one drops this assumption. 
We start with the following version of the Monotone Class Theorem.

\begin{lemma}\label{lem:monclass}
 Assume that $\CH$ is Dedekind $\sigma$-complete and $\CP\approx \ca_+^\sigma(\CH)$. Then, $L^\infty(\CP)\subset \CH$.
 \end{lemma}

The next proposition now shows that the only (generating) sublattice of $\Linfty$ satisfying the requirements (i) and (ii) above is $\Linfty$ itself.

\begin{proposition}\label{prop:equivalences0}
  The following statements are equivalent:
 \begin{itemize}
  \item[(1)] $\CH$ is Dedekind $\sigma$-complete and $\CP\approx \ca_+^\sigma(\CH)$.
  \item[(2)] $\CH$ is Dedekind $\sigma$-complete and $\ca(\CH)=\ca^\sigma(\CH)\cap \CH^*$.
  \item[(3)] $\CH$ is an ideal in $L^\Phi(\CP)$.
  \end{itemize}
  If $\CH\subset L^\infty(\CP)$, \tn{(1)--(3)} are furthermore equivalent to:
  \begin{itemize}
  \item[(4)] $\CH=L^\infty(\CP)$.
 \end{itemize}
 \end{proposition}
  
Considering $\CC$ instead of $\CH$ does not change the picture, since the closure of any (generating) sublattice of $\Linfty$ satisfying (i) and (ii) leads to the same Banach lattice, the closure of $\Linfty$.
 
\begin{proposition}\label{prop:equivalences}
 The following statements are equivalent:
 \begin{itemize}
  \item[(1)] $\CC$ is Dedekind $\sigma$-complete and $\CP\approx \ca_+^\sigma(\CC)$.
  \item[(2)] $\CC$ is Dedekind $\sigma$-complete and $\ca(\CC)=\ca^\sigma(\CC)$.
  \item[(3)] $\CC$ is an ideal in $L^\Phi(\CP)$.
  \end{itemize}
  If $\CH\subset L^\infty(\CP)$, \tn{(1)--(3)} are furthermore equivalent to
  \begin{itemize}
  \item[(4)]$\CC=\cl\big(L^\infty(\CP)\big)$.
   \end{itemize}
 \end{proposition}

\section{Applications}\label{sec.applications}

This section is devoted to {a more detailed discussion of the financial and economic implications of our theoretical results already mentioned in Section~\ref{main results and related}. }

\subsection{Closures of continuous functions}\label{sec:continuous}

Prominent sublattices of $L^\Phi(\CP)$ appearing in the literature -- at least for special cases of $\Phi$ -- are $\Norm_{L^\Phi(\CP)}$-closures of sets of continuous functions on a separable metrisable space $\Omega$. While {\cite{BionNadal} considers a general lattice of bounded continuous functions generating the Borel-$\sigma$-algebra and containing $\ind_\Omega$, other contributions consider} bounded Lipschitz functions, or bounded cylindrical Lipschitz functions, respectively, cf.\ \cite{Peng,HuWang}. The usual minimal assumption on $\CP$ is tightness. Sometimes one imposes that $\CP$ is convex and {weakly closed/compact}, cf.\ \cite{Beissner} and \cite[Section 3.3]{Viability}. In that case, $\CP$ has the stronger property of being countably  convex.

Throughout this subsection, we assume that $\Omega$ is a separable and metrisable topological space endowed with the Borel $\sigma$-algebra $\CF$. Let $C_b$ be the space of bounded {\em continuous} functions on $\Omega$, and let $\CH\subset C_b$ be a {generating} lattice containing $\ind_\Omega$. {We shall again work under Assumption~\ref{assumption}, which yields that $\iota\colon C_b\to L^\Phi(\CP)$ defined by $\iota(f)=[f]$ is a well-defined, continuous, and injective lattice homomorphism, cf.\ Theorem \ref{thm:L1}.} Abusing notation slightly, we refer to $\iota(C_b)$ as $C_b$, to the equivalence classes by capital letters though. As before, let  
\[\CC:=\cl\big(\CH\big),\]
endowed with $\Norm_{L^\Phi(\CP)}$ and the $\CP$-q.s.~order. 

Our first main observation is that the results in \cite{Beissner,BionNadal} are based on separability of the primal space, which holds under a comparatively mild tightness condition. The following result is a decisive generalisation of \cite[Proposition 2.6]{BionNadal}.

\begin{proposition}\label{prop:density}
Suppose that, for every $\ep>0$, there exists a compact set $K\subset \Om$ with
 \begin{equation}\label{eq:compact}
  \|\ind_{\Om\setminus K}\|_{L^\Phi(\CP)}<\ep.\end{equation}
 Then, $\CC$ is separable.
\end{proposition}

\begin{lemma}\label{lem:density}
Condition \eqref{eq:compact} is met in any of the following situations:
 \begin{itemize}
 \item[(1)] $\Omega$ is compact.
 \item[(2)] $\dom(\phi_{\Max})=[0,\infty)$ and, for all $t>0$, the set $\CP_t:=\{\P\in\CP\mid \phi_\P(t)>1\}$ is tight.
 \item[(3)] $\dom(\phi_{\Max})=[0,\infty)$ and $\CP$ is tight.
\end{itemize}
If $\Phi$ satisfies \eqref{eq:growth}, the validity of \eqref{eq:compact} implies that $\CP$ is tight.
\end{lemma}

We emphasise that (3) is the typical minimal assumption in the literature. It is, in particular, satisfied in the $G$-framework, see \cite[Theorem IV.2.5]{peng2019nonlinear}.

\begin{example}\
\begin{enumerate}[(1)]
\item Let $p\in[1,\infty)$, and consider the case, where $\phi_\P(x)=x^p$ for all $x\ge 0$ and $\P\in\CP$. Then, Lemma~\ref{lem:density} implies that \eqref{eq:compact} holds if and only if $\CP$ is tight. 
\item Consider the case of a doubly penalised robust Orlicz space as in Example~\ref{ex.doubpen}(1) with bounded multiplicative penalty $\theta\colon\CP\to(0,\infty)$ and additive penalty $\gamma\colon\CP\to[0,\infty)$. Then, $\dom(\phi_\Max)=[0,\infty)$ if and only if the joint Orlicz function $\phi$ satisfies $\dom(\phi)=[0,\infty)$. Moreover, condition (2) in Lemma~\ref{lem:density} is met if the lower level sets of $\gamma$ are tight. Notice that, in this case, the validity of \eqref{eq:growth} implies the boundedness of $\gamma$, and thus, naturally, the tightness of $\CP$. 
\end{enumerate}
\end{example}

\begin{remark}\
\begin{enumerate}[(1)]
\item Strictly speaking, Bion-Nadal \& Kervarec \cite{BionNadal} work with the \textit{Lebesgue prolongation} of a capacity $\mf c$ defined on a generating lattice of continuous functions. In most of their results, they assume that $\mf c$ is a \textit{Prokhorov capacity} on a separable metrisable space. As $\ind_{\Omega\setminus K}$ is l.s.c.\ for every compact $K\subset\Omega$, one thus obtains, for each $\ep>0$, the existence of a compact set $K\subset\Omega$ such that 
\[\mf c(\ind_{\Omega\setminus K})\le\ep.\]
This counterpart of \eqref{eq:compact} admits to perform our proof of Proposition~\ref{prop:density} in their framework, and the result transfers. 
\item We comment here on the role of Proposition~\ref{prop:equivalences} and Theorem~\ref{thm.super} in the present setting. {It is known that }
$C_b$ over a Polish space does not admit any nontrivial $\sigma$-order continuous linear functional when endowed with the pointwise order. One could therefore interpret Proposition~\ref{prop:equivalences} as a \textit{dichotomy}: either the closure $\CC$ of $C_b$ in $L^\Phi(\CP)$ behaves very much like the space of continuous functions, or it is an ideal of $L^\Phi(\CP)$, which could be obtained more directly as the closure of $\Linfty$ and to which in most typical cases Theorem~\ref{thm.super} applies. \\
As an illustrative example, consider $\Omega=[0,1]$ endowed with its $\sigma$-algebra $\CF$ of Borel sets and set $\CP$ to be the set of all atomless probability measures. Consider the robust weighted $L^1$-space, where $\theta\equiv 1$. One shows that each $X\in\CC$ has a unique continuous representative $f$ and satisfies $\|X\|_{L^\phi(\CP)}=\|f\|_\infty$. In this setting, the inclusions
\[\{0\}=\ca^\sigma(\CC)\subsetneq\ca(\CC)\subsetneq\CC^*\cap\ca\]
hold. For the first equality, note that $\CC$ is lattice-isometric to $C_b$, and the existence of a nontrivial $\sigma$-order continuous linear functional would contradict the result cited above. For the second strict inclusion, consider the linear bounded functional $\ell(X):=f(1)$, $X\in\CC$, where $f\in X$ is a continuous representative. Although it corresponds to the Dirac measure concentrated at 1, it cannot be identified with a measure absolutely continuous with respect to $\CP$. 
\end{enumerate}
\end{remark}

We conclude with a Riesz representation result for the dual of $\CC$, which follows directly from the more general observations in Section~\ref{sec.completion} and extends~\cite[Proposition 4]{Beissner} to our setting.

\begin{corollary}\label{cor:representation}
 Assume that $\CH=C_b$, $\CP$ is weakly compact, and that $\dom(\phi_\Max)=[0,\infty)$. Then, 
 $$\CC^*=\ca\big(\CC\big).$$ 
\end{corollary}

\subsection{Option spanning under uncertainty}\label{sec:options}

A rich strand of literature deals with the power of options to complete a market, at least in an approximate sense. 
{As announced in Section~\ref{discussion spanning}, we study option spanning under potentially nondominated uncertainty here. }

Fix a \textit{limited liability claim} $X\in L^\Phi(\CP)$, i.e., $X\succeq 0$ holds. Its \textit{option space} 
\[\CH_X:=\tn{span}\left(\{\ind_\Omega\}\cup\{(X-k\ind_\Omega)^+\mid k\in\R\}\right)\]
is the collection of all portfolios of call and put options written on $X$. In line with the simplifying assumption in Section~\ref{sec.completion}, we will assume w.l.o.g.\ that 
\[\CF=\sigma(\{f\mid f\in X\}),\]
a condition studied in detail in the existing literature on option spanning. We also introduce the norm closure 
\[\CC_X:=\cl(\CH_X),\]
the space of all contingent claims, which can be approximated by linear combinations of call and put options.

\begin{proposition}\label{prop:option}
$\big(\CH_X,\peq,\Norm_{L^\Phi(\CP)}\big)$ is a separable normed sublattice of $L^\Phi(\CP)$, and $\big(\mathcal \CC_X,\peq,\Norm_{L^\Phi(\CP)}\big)$
is a separable Banach sublattice of $L^\Phi(\CP)$. 
\end{proposition}

{By Proposition~\ref{prop:option}, Theorems~\ref{thm.main} and~\ref{thm.super} apply, and we may draw two interesting financial conclusions from our results.}
First, by Theorem~\ref{thm.main}, \textit{nondominated} uncertainty collapses both over the option space $\CH_X$ and its closure $\CC_X$. In fact, the same reference probability measure $\P^*$ can be chosen for both spaces (Corollary~\ref{cor.main}). $\P^*$ can be interpreted as intrinsic to and $J_{\P^*}(\CH_X)\subset L^1(\P^*)$ as a copy of the original option space $\CH_X$. This motivates the following corollary.
\begin{corollary}\label{cor:ae}
Let $\P^*\in\ca^1_+(L^\Phi(\CP))$ be a dominating probability measure for $\CH_X$ as constructed in Theorem~\ref{thm.main}. Then, for each $Y\in L^\Phi(\CP)$, there is a sequence $(Y_n)_{n\in\N}\subset\CH_X$ such that $Y_n\to Y$ $\P^*$-a.s.\ as $n\to\infty$. 
\end{corollary}
The second conclusion concerns the \textit{topological spanning power of $X$} and follows directly from Theorem~\ref{thm.super}.
\begin{corollary}\label{cor:spanningpower}Suppose $X$ has topological spanning power in that 
\begin{equation}\label{eq:options}\CC_X\text{ is an ideal of }L^\Phi(\CP).\end{equation}
Let $\P^*$ be a probability measure as in Theorem~\ref{thm.main}. Then, the following assertions hold:
\begin{itemize}
\item[(1)]$\CC_X$ is super Dedekind complete.
\item[(2)]$\ca(\CC_X)\approx\CP\approx \P^*$ and the unit ball of $\ca(\CC_X)$ is weakly compact in $L^1(\P^*)$.
\item[(3)]$\CC_X$ is lattice-isomorphic to an ideal of $L^1(\P^*)$. 
\item[(4)]$X\in\cl(\Linfty)=\CC_X$.
\end{itemize}
In particular, the topological spanning power of limited liability claims is always weaker than \eqref{eq:options} unless $\CP$ is dominated.
\end{corollary}

Moreover, under the mild growth condition \eqref{eq:growth} on $\Phi$, which does not depend on the concrete choice of the limited liability claim $X$ whatsoever, \eqref{eq:options} implies that all densities of priors in $\CP$ w.r.t.\ $\P^*$ are uniformly $\P^*$-integrable. This can be seen as a converse to the spanning power results on classical $L^p$-spaces, $1\le p<\infty$. In conclusion, the topological spanning power of options under nondominated uncertainty is \textit{always} weaker than \eqref{eq:options}, whereas $X$ always has full spanning power with respect to the reference measure $\P^*$ by Corollary~\ref{cor:ae}.

\subsection{Regular pricing rules and the Fatou property}\label{sec:regular}

\textit{Positive linear functionals} on a space of contingent claims are commonly interpreted as linear pricing rules. {In case of a} generating lattice $\CH\subset L^\Phi(\CP)$ as studied in Section~\ref{sec.completion}, $\sigma$-order continuity of a positive linear functional $\ell\colon\CH\to\R$ has the following economic interpretation. Whenever a sequence of contingent claims $(X_n)_{n\in\N}$ satisfies $X_{n+1}\peq X_n$, $n\in\N$, and $\inf_{n\in\N}X_n=0$, that is, the payoffs $X_n$ become arbitrarily invaluable in the objective $\CP$-q.s.\ order, their prices $\ell(X_n)$ under $\ell$ vanish:  
\[\lim_{n\to\infty}\ell(X_n)=0.\]
Pricing with such functionals does not exaggerate the value of (objectively) increasingly invaluable contingent claims.
By Lemma~\ref{lem:sigmaorder}, such functionals correspond to measures. The condition 
\[\CP\approx\ca^\sigma_+(\CH)\]
encountered in Lemma~\ref{lem:monclass} and Propositions~\ref{prop:equivalences0}--\ref{prop:equivalences} means that the set of all regular pricing rules (in the sense described above) holds the same information about {(im)possibility} of events as the set $\CP$ of ``physical priors". If $\CH=L^\Phi(\CP)$, Proposition~\ref{prop:equivalences} shows $\CP\subset\ca^\sigma(L^\Phi(\CP))=\ca(L^\Phi(\CP))$. 

For smaller generating lattices, our results describe a dichotomy in the case of full information: either the lattice does not generally admit aggregation even of countable order bounded families of contingent claims, or it is an ideal. In the case where bounded contingent claims are dense in $\CH$, the latter further specialises to $\CH=\Linfty$, or, if $\CH$ is closed in $L^\Phi(\CP)$, $\CH=\cl(\Linfty)$. 

{Another observation on the conjunction of $\CP\approx\ca^\sigma_+(\CH)$ and Dedekind $\sigma$-completeness of $\CH$ concerns the \emph{Fatou property}.} The latter is one of the most prominent phenomena studied in theoretical mathematical finance. {Morally speaking, it }relates order closedness properties of convex sets to dual representations of these in terms of measures. In a dominated framework, say, $L^\infty(\P)$ for a single reference measure $\P$, a subset $\CB\subset L^\infty(\P)$ is \textit{Fatou closed} if, for each sequence $(X_n)_{n\in\N}\subset\CB$ which converges $\P$-a.s.\ to some $X\in L^\infty(\P)$ and whose moduli are dominated by some $Y\in L^\infty(\P)$, the limit satisfies $X\in\CB$. Due to the super Dedekind completeness of $L^\infty(\P)$, Fatou closedness is equivalent to \textit{order closedness} of $\CB$. It is well known that $\CB\subset L^\infty(\P)$ is Fatou closed if and only if it has a representation of shape
\[\CB=\big\{X\in L^\infty\,\big|\, \forall\mu\in\mf D\colon\smallint X\,{\rm d}\mu\le h(\mu)\big\}\]
for a suitable set $\mf D\subset\ca(\P)$ and a function $h:\mf D\to\R$. 

This observation does not directly transfer to nondominated frameworks. {In our setting, each $\mu\in\ca^\sigma_+(\CH)$ defines a seminorm $\rho_\mu:\CH\to[0,\infty)$ by $\rho_\mu(X)=\int|X|\,{\rm d}\mu$. Let $\tau$ be the locally convex topology on $\CH$ generated by them.\footnote{~In the literature, $\tau$ is often referred to as the \emph{absolute weak topology} $|\sigma|(\CH,\ca^\sigma_+(\CH))$.} }
\begin{observation}
If $\CP\approx\ca^\sigma_+(\CH)$, then $\tau$ is a locally convex Hausdorff topology.  In particular, $\tau$ admits the application of separating hyperplane theorems. By Kaplan's Theorem~\cite[Theorem 3.50]{AliBurk2} and \cite[Theorem 1.57]{AliBurk2}, the dual of $(\CH,\tau)^*$ is $\ca^\sigma(\CH)$.
\end{observation}
Let $\mathcal B\subset\CH$ be a nonempty convex set. 
\begin{observation}
If $\CH$ is Dedekind $\sigma$-complete, the following statements are equivalent:
\begin{enumerate}[(1)]
\item $\mathcal B$ is sequentially order closed.
\item For all sequences $(X_n)_{n\in\N}$ whose moduli $(|X_n|)_{n\in\N}$ admit some upper bound $Y\in\CH$ and which converge $\CP$-q.s.\ to some $X\in\CH$, this limit $X$ lies in $\mathcal B$.
\end{enumerate}
\end{observation}
Note that (2) is the direct counterpart of Fatou closedness as formulated above. 

{Hence, in the outlined situation, there is a locally convex Hausdorff topology $\tau$ on $\CH$ such that each $\tau$-closed convex set $\mathcal B\subset\CH$ is sequentially order closed. The assumptions in the preceding two observations appear to be analytic minimal requirements for a fruitful study of the (sequential) Fatou property. In that case, Proposition~\ref{prop:equivalences0} states that studying the (sequential) Fatou property only makes sense on \textit{ideals} of robust Orlicz spaces.}

However, in robust frameworks, there is a caveat concerning the Fatou property: sequences are usually not sufficient to unfold its full analytic power, and intuitive reasoning learned in dominated frameworks usually fails. We refer to \cite{Maggisetal} for a detailed discussion of this issue.

\subsection{Utility theory for multiple agents}\label{sec:utility}

{As anticipated in Section~\ref{main results and related}, robust Orlicz spaces are canonical model spaces for aggregating \emph{(variational) preferences} of a cloud of agents. This aggregation procedure is related to several economic problems.
\begin{itemize}
\item The preference relation of each of these agents may reflect the opinion of a Bayesian or non-Bayesian expert as in \cite{Amarante}, see also the references therein. Studying unanimity, i.e.\ agreement of such experts in their judgement, is crucial for appropriate theory-building.  
\item In \cite{Viability}, it is established that arbitrage is absent in a market if and only if the latter is viable (i.e.\ prices in a security market are result of an equilibrium of---potentially many---agents). In contrast to the preceding literature, Knightian uncertainty is explicitly taken into account and often forces the consideration of a whole cloud of representative agents instead of a single representative agent.   
\item Regarding the numerical representation of {\em incomplete} preferences, these usually require considering a whole cloud of agents. The incomplete preference relation in question is then interpreted as the unanimous order, cf.\ \cite{Incomplete} and the references therein. 
\end{itemize}
The preceding list is incomplete, of course.} {\em Variational preferences} encompass other prominent classes of preferences, such as multiple prior preferences of Gilboa \& Schmeidler and the multiplier preferences of Hansen \& Sargent, cf.\ the discussion in \cite{Variational}. One of their most appealing qualities is the handy separation of risk attitudes (measured by the prior-wise expected utility approach) and ambiguity or uncertainty attitudes (as expressed by the choice of the underlying set of priors and their additive penalisation). The economic problems mentioned above may require to consider more than one utility function though. 
{
\begin{assumption}Throughout this subsection we impose the following assumptions:
\begin{enumerate}[(1)]
\item $\CI$ is a nonempty set of agents.
\item Each agent $i\in\CI$ has preferences over $\mathcal L^\infty(\Omega,\CF)$, the space of bounded real-valued random variables.
\item The preferences of each agent $i\in\CI$ are captured by a utility function
 \[\mf U_i\colon\mathcal L^\infty(\Omega,\CF)\to\R,\quad f\mapsto \inf_{\P\in\CP_i}\E_\P[u_i(f)]+c_i(\P).\]
Here, $\CP_i$ is a nonempty set of probability measures on $(\Omega,\CF)$ equivalent to a reference probability measure $\P_i^*$; $c_i\colon \CP_i\to[0,\infty)$ satisfies $\inf_{\P\in\CP_i}c_i(\P)=0$; and the scalar utility function $u_i\colon \R\to\R$ is concave, nondecreasing, nontrivial in that $u_i\not\equiv 0$, and satisfies $u_i(0)=0$. 
\item For $f,g\in\mathcal L^\infty(\Omega,\CF)$, we set $f\peq_i g$ iff $\mf U_i(f)\le \mf U_i(g)$. 
\end{enumerate}
\end{assumption}
For more details on such preferences, we refer to \cite{Variational}.} By an affine transformation, we can w.l.o.g.\ assume that 
\[\mf U_i(-\ind_\Omega)=\inf_{\P\in\CP_i}u_i(-1)+c_i(\P)=-1,\quad\text{for all }i\in\CI.\]

Aggregating the preferences of all agents in $\CI$ leads to the (potentially incomplete) \emph{unanimous preference relation}
\[f\tle g\quad:\iff\quad \forall\,i\in\CI:~f\peq_ig,\]
{which obviously possesses a multi-utility representation in terms of the family $(\mf U_i)_{i\in\CI}$.}
For its closer study, we observe that for $\CP:=\bigcup_{i\in\CI}\CP_i$, $f=g$ $\CP$-q.s.\ implies $f\tle g$ and $g\tle f$, i.e., all agents are indifferent between $f$ and $g$. Hence, we may consider the preference relations $\peq_i$ on the space $L^\infty(\CP)$ instead without losing any information. The definition of $\mf U_i$ on $L^\infty(\CP)$ is immediate. 

Next, for $\P\in\CP$, consider 
\[\phi_\P(x):=\sup_{i\in\CI:\,\P\in\CP_i}\frac{-u_i(-x)}{1+c_i(\P)},\quad\text{for }x\in[0,\infty).\]
$\phi_\P$ is convex, lower semicontinuous, nondecreasing, and satisfies $\phi_\P(0)=0$. We shall impose the condition that $\phi_\P$ is an Orlicz function. Moreover, one easily obtains 
\[\phi_\P(1)\le 1,\quad\text{for all }\P\in\CP.\]
Set $\Phi:=(\phi_\P)_{\P\in\CP}$ and consider the associated robust Orlicz space which satisfies $\Linfty\subset L^\Phi(\CP)$. 

We claim that $L^\Phi(\CP)$ is a canonical maximal model space to study all individual preferences  $\peq_i$, $i\in\CI$, simultaneously. This is due to the observation that each preference relation $\peq_i$ canonically extends to a \textit{continuous} preference relation on $L^\Phi(\CP)$. Indeed, note that by concave duality, for each $i\in\CI$, there is a set $\mbf M_i\ll\CP$ of finite measures and a non-negative function $h_i:\mbf M_i\to[0,\infty)$ such that 
\[\mf U_i(X)=\inf_{\mu\in\mbf M_i}\int X\,{\rm d}\mu+h_i(\mu),\quad\text{for all }X\in\Linfty.\]
Moreover, one can show that, for each $i\in\CI$, all $X\in\Linfty$, and all $\mu\in\mbf M_i$, 
\[-1\le \mf U_i\left(-\|X\|_{L^\Phi(\CP)}^{-1}|X|\right)\le-\|X\|_{L^\Phi(\CP)}^{-1}\int|X|\,{\rm d}\mu+h_i(\mu).\]
From this, we infer 
\[\int |X|\,{\rm d}\mu\le (1+h_i(\mu))\|X\|_{L^\Phi(\CP)},\quad\text{for all }X\in\Linfty,\,i\in\CI,\text{ and }\mu\in\mbf M_i.\]
By monotone convergence, the same estimate holds for all $X\in L^\Phi(\CP)$. Hence, setting
\[\mf U_i^\sharp(X):=\inf_{\mu\in\mbf M_i}\int X\,{\rm d}\mu+h_i(\mu),\quad\text{for all }X\in L^\Phi(\CP),\]
and 
\[X\peq_i Y\quad:\iff\quad\mf U_i^\sharp(X)\le\mf U_i^\sharp(Y),\quad\text{for }X,Y\in L^\Phi(\CP),\]
we have extended the initial preference relations to $L^\Phi(\CP)$ in a continuous manner. 

In case that all agents have the same attitude towards \textit{risk}, i.e., the utility function $u_i$ does not depend on $i$, it is straightforward to construct examples where Theorem~\ref{thm:equality} is applicable and we have the identity 
\[L^\Phi(\CP)=\mf L^\Phi(\CP),\]
i.e., $L^\Phi(\CP)$ is the model space for the minimal agreement among all agents under consideration on which well-defined utility can be attached to all objects.\footnote{~A situation where $u_i$ does not depend on $i$ is, for instance, the aggregation of opinions of {\em Bayesian} experts in \cite{Amarante}.}

\section{Conclusion}\label{sec.conclusion}
 Our results highlight both the advantages and the cost of taking a top-down or a bottom-up approach to robust Orlicz spaces, respectively. Whereas the former may lack good dual behaviour, it has reasonable order completeness properties and reflects the full nondominated nature of the underlying uncertainty structure. The latter may be handy analytically, but either ignores the nondominated uncertainty structure \textit{a posteriori}, or tends to lead to a complete breakdown of almost all lattice properties.

\begin{appendix}

\section{Proofs of Section \ref{sec.orlicz}}\label{app.A}

\begin{proof}[Proof of Proposition~\ref{prop:Banach}]
The fact that $L^\Phi(\CP)$ is an ideal of $L^0(\CP)$ follows directly from the fact that each $\phi_\P$ is nondecreasing and convex and the fact that the supremum is subadditive. This proves (2). 
 
In order to prove (1), the ideal property of $L^\Phi(\CP)\subset L^0(\CP)$ and Dedekind $\sigma$-completeness of the latter space with respect to the $\CP$-q.s.\ order imply that $L^\Phi(\CP)$ is Dedekind $\sigma$-complete. In a similar way, it follows that $\Norm_{L^\Phi(\CP)}$ defines a norm on $L^\Phi(\CP)$. 
Let $(X_n)_{n\in\N}$ be a Cauchy sequence. Notice that, since $\phi_\P$ is convex and nontrivial for all $\P\in \CP$, there exist $a_\P> 0$ and $b_\P\geq 0$ such that
 \begin{equation}\label{eq.convlin}
  \phi_\P(x)\geq(a_\P x-b_\P)^+,\quad \text{for all }x\geq 0.
 \end{equation}
 By possibly passing to a subsequence, we may assume that 
 \[\|X_n-X_{n+1}\|_{L^\Phi(\CP)}<4^{-n},\quad \text{for all }n\in\N.\]
For all $n\in \N$, let $\la_n>0$ with $\|X_n-X_{n+1}\|_{L^\Phi(\CP)}< \la_n\leq 4^{-n}$. In particular, $\lambda_n^{-1}2^{-n}\ge 2^n$, i.e.\ we can fix $n_\P\in\N$ such that $a_\P\lambda_n^{-1}2^{-n}-b_\P>0$ holds for all $n\ge n_\P$.
Markov's inequality together with equation \eqref{eq.convlin} shows, for all $\P\in\CP$,
 \begin{align*}
  \sum_{n=n_\P}^\infty\P\big(|X_n-X_{n+1}|\ge 2^{-n}\big)&\le\sum_{n=n_\P}^\infty\P\left(\big(a_\P(\la_n^{-1}|X_n-X_{n+1}|)-b_\P\big)^+\ge \big(a_\P\la_n^{-1} 2^{-n}-b_\P\big)^+\right)\\
  &\leq \sum_{n=n_\P}^\infty\left(a_\P2^n-b_\P\right)^{-1}\E_\P\left[\phi_\P\left(\la_n^{-1}|X_n-X_{n+1}|\right)\right]\\
  &\leq \sum_{n=n_\P}^\infty\frac{1}{a_\P2^n-b_\P}<\infty.
 \end{align*}
Applying the Borel-Cantelli Lemma yields that
 \[\P\big(|X_n-X_{n+1}|\leq 2^{-n}\tn{ eventually}\big)=1,\quad\tn{for all }\P\in\CP.\]
 Hence, the event $\Omega^*:=\{\lim_{n\to \infty} X_n\tn{ exists in }\R\}\in\CF$ satisfies $\P(\Omega^*)=1$ for all $\P\in\CP$. We set $X$ to be (the equivalence class in $L^0(\CP)$ induced by) $\limsup_{n\to\infty}X_n$. Now, let $\P\in \CP$ and $\al>0$ be arbitrary. Choose $k\in\N$ such that $\sum_{i\ge k}\lambda_i\al\le 1$.  
 For $l>k$, we can estimate 
 \begin{align*}
  \phi_\P(\al|X_{n_k}-X_{n_l}|)&\le\phi_\P\left(\sum_{i=k}^{l-1}\al|X_{n_{i+1}}-X_{n_i}|\right)\le\sum_{i=k}^{l-1}\lambda_i\al\phi_\P\left(\lambda_i^{-1}|X_{n_{i+1}}-X_{n_i}|\right)\le\sum_{i=k}^\infty\lambda_i\al.
 \end{align*}
 Notice that the last bound is uniform in $l$ and $\P$. Letting $l\to\infty$ and using lower semicontinuity of $\phi_\P$, 
 \begin{align*}
  \phi_\P(\al|X_{n_k}-X|)&\le\sum_{i=k}^\infty\lambda_i\al.
 \end{align*}
 This implies
 \[\limsup_{k\to\infty}\sup_{\P\in\CP}\E_\P[\phi_\P\left(\al|X_{n_k}-X|\right)]\le \lim_{k\to\infty}\sum_{i=k}^\infty\lambda_i\al=0.\]
 As $\al>0$ was arbitrary, $X\in L^\Phi(\CP)$ and $\lim_{k\to\infty}\|X_k-X\|_{L^\Phi(\CP)}=0$ follow. 
 
 For (3), let $X\in L^\Phi(\CP)$. By lower semicontinuity of $\phi_\P$ and Fatou's Lemma, for all $\P\in \CP$,
 \[
  \E_\P\big[a_\P\tfrac{|X|}{\|X\|_{L^\Phi(\CP)}}-b_\P\big]\leq \E_\P\big[\phi_\P(\|X\|_{L^\Phi(\CP)}^{-1} |X|)\big]\leq 1,
 \]
 showing that $\E_\P[|X|]\leq \frac{1+b_\P}{a_\P}\|X\|_{L^\Phi(\CP)}$, that is, \eqref{eq:normP}. At last, (4) is a direct consequence of~\eqref{eq:normP}.
\end{proof}

\begin{proof}[Proof of Theorem~\ref{thm:L1}]
{(1) implies (2):}
 Suppose $\Linfty\subset L^\Phi(\CP)$. Then, we can find some $\al>0$ such that 
 \[\sup_{\P\in\CP}\phi_\P(\al)= \sup_{\P\in\CP}\E_\P[\phi_\P(\al\ind_\Omega)]\leq 1.\]
 {(2) implies (3): Let} $\al>0$ with $\phi_\Max(\al)=\sup_{\P\in \CP}\phi_\P(\al)< \infty$. Since $\phi_\Max$ is convex, we may w.l.o.g.\ assume that $\phi_\Max(\al)\le 1$. For $\P\in \CP$ and $\mu\in \ca_+(\P)$, let
 \[
  \|\mu\|_{\P}':=\sup\bigg\{\int|X|\,{\rm d}\mu\,\bigg|\, \|X\|_{L^{\phi_\P}(\P)} =1\bigg\}.
 \]
 Then, by \cite[Theorem 2.6.9 \& Corollary 2.6.6]{MeyNie},\footnote{~The cases $L^{\phi_\P}(\P)\in \{L^1(\P),L^\infty(\P)\}$ are not treated in this reference, but equation~\eqref{eq:koethe} is well known for them.}
 \begin{equation}\label{eq:koethe}
  \|X\|_{L^{\phi_\P}(\P)}=\sup\bigg\{\int|X|\,{\rm d}\mu\, \bigg| \, \mu\in\ca_+(\P),\, \|\mu\|_{\P}' = 1\bigg\},\quad \text{for all }\P\in \CP\text{ and }X\in L^{\phi_\P}(\P).
 \end{equation}
 Since $\sup_{\P\in \CP}\phi_\P(\al)\leq 1$, $\|\ind_\Omega\|_{L^{\phi_\P}(\P)}\le\alpha^{-1}$. Hence, for all $\mu\in\ca_+(\P)$ with $\|\mu\|_\P'=1$, 
 \begin{equation}\label{eq:bound1}
  \mu(\Omega)=\|\ind_\Omega\|_{L^{\phi_\P}(\P)}\int\Big(\|\ind_\Omega\|_{L^{\phi_\P}(\P)}\Big)^{-1}\ind_\Omega\,{\rm d}\mu\le\frac{1}\al.
 \end{equation}
 For $\P\in \CP$, let 
 \[\CQ_\P:=\Big\{\tfrac{1}{\mu(\Omega)}\mu\,\Big|\,\mu\in\ca_+(\P),\,\|\mu\|_\P'=1\Big\}.\]
 By \eqref{eq:normP}, $\P\in\CQ_\P$ holds for all $\P\in\CP$. 
 We also define 
 \[\CQ:=\big\{\QW\in \ca_+^1(\CP)\,\big|\, \exists\,\P\in\CP:~\QW\in\CQ_\P\big\}.\]
 Fix $\Q\in\CQ$, let $\P\in\CP$ such that $\Q\in\CQ_\P$, and let $\mu\in\ca_+(\P)$ such that $\Q=\mu(\Omega)^{-1}\mu$. Then, \eqref{eq:bound1} implies that
 \[\|\Q\|_\P'=\mu(\Omega)^{-1}\ge\al.\]
 The function
  \[\theta(\QW):=\frac{1}{\inf_{\P\in\CP\colon \QW\in\CQ_\P}\|\Q\|_\P'},\quad \text{for }\Q\in\CQ,\]
is thus bounded and takes positive values. Moreover, for $X\in L^0(\CP)$,
 \begin{align*}
  \|X\|_{L^\Phi(\CP)}&=\sup_{\P\in\CP}\|X\|_{L^{\phi_\P}(\P)}=\sup_{\P\in\CP}\sup_{\Q\in\CQ_\P}\frac{1}{\|\Q\|_\P'}\E_\Q[|X|]\\
  &=\frac{1}{\inf_{\P\in\CP\colon \QW\in\CQ_\P}\|\Q\|_\P'}\sup_{\Q\in\CQ}\E_\Q[|X|]=\sup_{\Q\in\CQ}\theta(\Q)\E_\Q[|X|].
 \end{align*}
 (4) is a direct consequence of (3) if we set $\kappa:=\sup_{\Q\in \CQ} \theta(\Q)$ or, equivalently, $\kappa:=\|\ind_\Omega\|_{L^\Phi(\CP)}$.
 
 (4) clearly implies (1). 
\end{proof}

\begin{proof}[Proof of Theorem \ref{thm.main}]
The separability of $\CH$ implies that the unit ball $\big\{\ell\in\CH^*\, \big|\, \|\ell\|_{\CH^*}\le 1\big\}$ endowed with the weak* topology is compact, metrisable, and thus separable, cf.~\cite[Theorem 3.16]{MR1157815}. Hence, the set
\[
 \big\{ \tfrac1{\|\P\|_{\CH^*}}\P\, \big|\, \P\in \CP\big\}\subset \big\{\ell\in\CH^*\, \big|\, \|\ell\|_{\CH^*}\le 1\big\}
\]
 is separable, and there exists a sequence $(\P_n)_{n\in \N}$ such that, for all $X\in \CH$, $\sup_{n\in \N} \E_{\P_n}[|X|]>0$ holds if and only if $X\neq 0$.
Consider the measure 
\[\mu^*:=\sum_{n\in \N}2^{-n}\min\{1,\|\P_n\|_{\CH^*}^{-1}\}\P_n\in\CH^*,\]
which satisfies $\mu^*(\Omega)\le 1$. For $s>0$ appropriately chosen, the probability measure $\P^*:=s\mu^*\in\CH^*$ is a countable convex combination of $(\P_n)_{n\in\N}$, and the functional $\P^*$ is strictly positive by construction. Hence, for $X,Y\in \CH$, $X\peq Y$ if and only if $\E_{\P^*}[(Y-X)^-]=0$, which immediately proves that the canonical projection $J_{\P^*}\colon \CH\to L^1(\P^*)$ is injective. By construction, we see that $\P^*\in \CP$ if $\CP$ is countably convex.
\end{proof}

\begin{proof}[Proof of Corollary \ref{cor.main}]
 As in the proof of the previous theorem, we see that the set
\[
 \big\{ \theta (\Q)\cdot \Q \, \big|\, \Q\in \CQ\big\}\subset \big\{\ell\in\CH^*\, \big|\, \|\ell\|_{\CH^*}\le 1\big\}
\]
is separable with respect to the relative weak* topology. Hence, there exists a countable family $(\Q_n)_{n\in \N}\subset \CQ$ such that, for all $X\in \CH$,
\[
 \sup_{n\in \N} \theta (\Q_n)\|X\|_{L^1(\Q_n)}=\sup_{\Q\in \CQ} \theta (\Q)\|X\|_{L^1(\Q)}=\|X\|_{L^\Phi(\CP)}.
\]
\end{proof}

\begin{proof}[Proof of Proposition \ref{prop:equality}]
(2) clearly implies (1). Now suppose that (1) holds. By Theorem~\ref{thm:L1}, we have 
 \[
  \|X\|_{L^\Phi(\CP)}=\sup_{\mu\in\mf D}\int|X|\,{\rm d}\mu,
 \]
 where $\mf D:=\{\mu\in\ca_+(L^\Phi(\CP))\mid \|\mu\|_{L^\Phi(\CP)^*}\le 1\}$. In particular, $\sup_{\mu\in\mf D}\mu(\Omega)<\infty$ holds because of the assumption $\Linfty\subset L^\Phi(\CP)$. Set 
 \begin{align*}
 \mf R&:=\{\mu(\Omega)^{-1}\mu\mid \mu\in\mf D\}\subset\ca^1_+(L^\Phi(\CP)),\\
 \psi_\Q(x)&:=\mu(\Omega)x,~x\ge 0,~\tn{for }\Q=\mu(\Omega)^{-1}\mu\in \mf R,\\
 \Psi&=(\psi_\Q)_{\Q\in\mf R}.
 \end{align*}
 Then, $L^\Phi(\CP)\subset\mf L^\Psi(\mf R)$ holds by construction. Suppose now that $X\in L^0(\CP)\setminus L^\Phi(\CP)$. Then, we must be able to find a sequence $(\mu_n)_{n\in\N}\subset\mf D$ such that $\mu_n|X|\ge 2^{n}$, $n\in\N$. By the Banach space property of $L^\Phi(\CP)^*$, $\mu^*:=\sum_{n=1}^\infty2^{-n}\mu_n\in\mf D$, and we observe 
 \[
  \int|X|\,{\rm d}\mu^*=\sum_{n=1}^\infty2^{-n}\mu_n|X|\ge\sum_{n=1}^\infty 1=\infty.
 \]
 This completes the proof of the identity $L^\Phi(\CP)=\mf L^\Psi(\mf R)$. \\
 Consider now the special case of $\CP$ being countably convex and \eqref{eq.thm:equality} being satisfied. 
 Observe that, for all $\al>0$, $\P\in\CP$, and all $X\in L^0(\CP)$, 
\[\E_\P[\phi_\P(\al|X|)]\le\E_\P[\phi_\Max(\al|X|)]\le\E_\P[\phi_\P(\al c_\P|X|)].\]
If we set $\Psi=(\phi_\Max)_{\P\in\CP}$, this is sufficient to prove the following chain of inclusions:
\[L^\Psi(\CP)\subset L^\Phi(\CP)\subset\mf L^\Phi(\CP)=\mf L^\Psi(\CP).\]
The proof is complete if we can show $\mf L^\Psi(\CP)\subset L^\Psi(\CP)$. To this end, 
let $X\in L^0(\CP)\setminus L^\Phi(\CP)$. Then, there exists a sequence $(\P_n)_{n\in \N}\subset \CP$ with
 \[
  \|X\|_{L^{\phi_{\Max}}(\P_n)}>2^nn,\quad \text{for all }n\in \N.
 \]
 Define $\P:=\sum_{n\in \N}2^{-n}\P_n\in\CP$ (because $\CP$ is countably convex), and let $s>0$ be arbitrary. Then, 
 \[
\E_\P[\phi_\Max(s|X|)]=\sum_{n=1}^\infty 2^{-n}\E_{\P_n}\left[\phi_{\Max}(s|X|)\right]\ge\sum_{n=1}^\infty \E_{\P_n}\left[\phi_{\Max}(2^{-n}s|X|)\right]=\infty,
 \]
 which proves that $X\notin\mf L^\Psi(\CP)$.
\end{proof}

\begin{proof}[Proof of Theorem \ref{thm:equality}]
Let $X\in L^0(\CP)\setminus L^\Phi(\CP)$.
Then, there is a sequence 
$(\P_n)_{n\in\N}\subset\CP$ such that, for all $n\in\N$, $\|X\|_{L^{\phi_{\P_n}}(\P_n)}>2^{2n}$, which in particular entails
\[\E_{\P_n}[\phi(\theta(\P_n)2^{-n}|X|)]>2^n(1+\gamma(\P_n)).\]
Fix $\P^*\in\CP$ and consider the measure
\[\Q:=\sum_{n=1}^\infty2^{-n}\big(\tfrac{\gamma(\P_n)}{1+\gamma(\P_n)}\P^*+\tfrac{1}{1+\gamma(\P_n)}\P_n\big).\]
By convexity of $\gamma$ and the countable convexity of its lower level sets, $\gamma(\Q)\le\gamma(\P^*)+1$, $n\in\N$. For $\al>0$ arbitrary, set $I:=\{n\in\N\mid \theta(\Q)\al\ge \theta(\P_n)2^{-n}\}$, an infinite set. Then, 
\begin{align*}\E_{\Q}[\phi(\theta(\Q_n)\alpha|X|)]&\ge \sum_{n\in I}\tfrac{1}{2^n(1+\gamma(\P_n)}\E_{\P_n}[\phi(\theta(\P_n)2^{-n}|X|)]=\infty.
\end{align*}
This proves $\|X\|_{L^{\phi_{\Q}}(\Q)}=\infty$, which means $X\notin\mf L^\Phi(\CP)$. 
\end{proof}

\section{Proofs of Section \ref{sec.completion}}\label{app.B}

\begin{proof}[Proof of Lemma~\ref{lem:sigmaorder}]
Let $\mathcal Y$ denote the real vector space of all $\sigma$-order continuous linear functionals on $\CH$. As $\CH$ is a vector lattice, $\CY$ is a vector lattice itself when endowed with the order 
\[\ell\peq^*\ell'\quad:\iff\quad \forall\,X\in\CH,\,X\peq 0:~\ell(X)\le \ell'(X),\]
cf.\ \cite[Theorem 1.57]{AliBurk2}. As such, for each $\ell\in\CY$ there are unique $\ell^+,\ell^-\succeq^* 0$ such that $\ell=\ell^+-\ell^-$. We may hence assume for the moment that $\ell\succeq^* 0$. \\
Then, for each sequence $(X_n)_{n\in\N}\subset\CH$ possessing representatives $(f_n)_{n\in\N}$ such that $f_n\downarrow 0$ holds pointwise, $\inf_{n\in\N}X_n=0$ holds in $\CH$. Consider the vector lattice
\[\CL:=\big\{f\in\CL^0(\Omega,\CF)\, \big|\, [f]\in\CH\big\}\]
and the linear functional $\ell_0\colon\CL\to\R$ defined by $\ell_0(f)=\ell([f])$. Then, $\ell_0(f_n)\downarrow 0$ for all sequences $(f_n)_{n\in\N}\subset\CL$ such that $f_n\downarrow 0$ pointwise.
Since, by our assumption on $\CH$, $\CF=\sigma(\CL)$, \cite[Theorem 7.8.1]{Bogachev} provides a unique finite measure $\mu$ on $(\Omega,\CF)$ such that 
\[\ell_0(f)=\int f\,{\rm d}\mu,\quad\text{for all } f\in \CL.\]
As $|f|\in\CL$ for all $f\in\CL$, each $f\in\CL$ is $\mu$-integrable. Moreover, for all $X\in\CH$ and $f,g\in X$, 
\[\int f\,{\rm d}\mu=\ell_0(f)=\ell(X)=\ell_0(g)=\int g\,{\rm d}\mu.\]
In particular, considering that $\ind_N\in\CL$ for all $N\in\CF$ satisfying $\sup_{\P\in\CP}\P(N)=0$, $\mu\in\ca_+(\CP)$ follows.\\
Finally, for a general $\ell\in\CY$, let $\nu,\eta\in\ca_+(\CP)$ be the finite measures corresponding to $\ell^+$ and $\ell^-$, respectively. Setting $\mu:=\nu-\eta$, we obtain for all $X\in\CH$ that
\[\ell(X)=\ell^+(X)-\ell^-(X)=\int X\,{\rm d}\nu-\int X\,{\rm d}\eta=\int X\,{\rm d}\mu.\]
Moreover, the total variation measure $|\mu|$ satisfies $\int |f|\,{\rm d}|\mu|\le \int |f|\,{\rm d}(\nu+\eta)<\infty$, $f\in\CL$. \\
At last, suppose that the representing signed measure of $\ell\in\CY$ is a measure. Then, $\ell\succeq^* 0$ holds automatically, and the proof is complete. 
\end{proof}

\begin{proof}[Proof of Lemma~\ref{lem:representation00}]
Let $\ell\in\CH^*$. In order to verify $\ell\in\ca^\sigma(\CH)$, let $X,Y\in\CH$. Then, all $Z\in\CH$ with the property $X\peq Z\peq Y$ satisfy $|Z|\le X^-+Y^+$. We obtain 
\begin{align*}\sup\{\ell(Z)\mid X\peq Z\peq Y\}&\le\sup\{\|\ell\|_{\CH^*}\|Z\|_{L^\Phi(\CP)}\mid X\peq Z\peq Y\}\\
&\le\|\ell\|_{\CH^*}\big(\|X^-\|_{L^\Phi(\CP)}+\|Y^+\|_{L^\Phi(\CP)}\big)<\infty.
\end{align*}
This gives condition (i) in Definition~\ref{def:sigmaorder}. The validity of condition (ii) is a direct consequence of the lattice norm property and $\sigma$-order continuity of $\Norm_{L^\Phi(\CP)}$ on $\CH$. The inclusion $\CH^*\subset\ca^\sigma(\CH)$ together with Lemma~\ref{lem:sigmaorder} implies {$\CH^*=\ca(\CH)$}.
\end{proof}

\begin{proof}[Proof of Proposition~\ref{lem:representation0}]
By \cite[Proposition 1.2.3(ii)]{MeyNie}, the closure $(\CC,\peq)$ of the sublattice $(\CH,\peq)$ of $L^\Phi(\CP)$ is a sublattice as well. As $\Norm_{L^\Phi(\CP)}$ is a lattice norm on $\CC$, $\big(\CC,\peq,\Norm_{L^\Phi(\CP)}\big)$ is a Banach lattice by construction.\\
The inclusion $\ca(\CC)\subset\ca(\CH)$ is trivial. For the converse inclusion, let $\mu \in \ca(\CH)$, i.e., $|\mu|\in\CH^*$.
Since $\CH$ is dense in $\CC$, there exists a unique $\ell\in \CC^*$ with
$$\ell(X)=\int X\,{\rm d}|\mu|,\quad \text{for all }X\in \CH.$$
Let $X\in \CC\cap \Linfty$. Then, by Proposition \ref{prop:Banach}, there exists a sequence $(X_n)_{n\in \N}\subset \CH\cap \Linfty$ with $\sup_{n\in \N}\|X_n\|_{\Linfty}<\infty$, $\|X-X_n\|_{L^\Phi(\CP)}\to 0$, and $X_n\to X$ $\CP$-q.s.~as $n\to \infty$. Since $|\mu|\ll \CP$, dominated convergence implies
\[
 \int X\,{\rm d}|\mu|=\lim_{n\to \infty}\int X_n\,{\rm d}|\mu|=\lim_{n\to \infty}\ell (X_n)=\ell (X).
\]
 Now, let $X\in \CC$ arbitrary. Then, 
\begin{align}\begin{split}\label{eq:approximation}\int |X|\,{\rm d}|\mu| &= \sup_{n\in \N}\int \big(|X|\wedge n\ind_\Omega\big)\,{\rm d}|\mu|=\sup_{n\in \N}\ell\big(|X|\wedge n\ind_\Omega\big)\\
&\leq \sup_{n\in \N}\|\ell\|_{\CC^*}\big\||X|\wedge n\ind_\Omega\big\|_{L^\Phi(\CP)}\leq \|\ell\|_{\CC^*}\|X\|_{L^\Phi(\CP)}.
\end{split} \end{align}
From this observation, the equality $|\mu|=\ell$ follows, which is sufficient to prove that $\mu\in \ca(\CC)$. \\
The remaining assertions easily follow with $\CC^*=\CH^*$ and Lemma~\ref{lem:representation00}.
\end{proof}

\begin{proof}[Proof of Lemma \ref{lem.Linfty}]
{(2) equivalent to (3):} For $X\in L^\Phi(\CP)$, notice that
 \[
  \sup_{\P\in \CP}\E_\P\big[\phi_\P(\al |X|)\mbf 1_{\{|X|> n\}}\big]\to 0,\quad \text{as } n\to \infty
 \]
 for all $\al>0$ is equivalent to $\|X\mbf 1_{\{|X|> n\}}\|_{L^\Phi(\CP)}\to 0$ as $n\to \infty$.
 
 {(3) implies (1):
 Set $X_n:=X\ind_{\{|X|\le n\}}\in\Linfty$, $n\in\N$. If $\|X\mbf 1_{\{|X|> n\}}\|_{L^\Phi(\CP)}\to 0$ as $n\to \infty$, it follows that $\|X-X_n\|_{L^\Phi(\CP)}\to 0$ as $n\to \infty$.}

{(1) implies (3): Assume that $X\in \cl\big(\Linfty\big)$.} Let $(Y_n)_{n\in \N}\subset \Linfty$ with $\|X-Y_n\|_{L^\Phi(\CP)}\to 0$ as $n\to \infty$. Let $X_m:=(X\wedge m\ind_\Omega)\vee (-m\ind_\Omega)$ for all $m\in \N$. Then, for all $m,n\in \N$ with $m\geq\|Y_n\|_{L^\infty(\CP)}$, it follows that
 \[
  |X-X_m|\leq |X-Y_n|,
 \]
 which implies that $\|X-X_m\|_{L^\Phi(\CP)}\to 0$ as $m\to \infty$. Finally notice that
 \begin{align*}
  |X|\mbf1_{\{|X|>2m\}}&=(|X|-m\ind_\Omega)\mbf 1_{\{|X|>2m\}}+m \mbf 1_{\{|X|>2m\}}\leq 2\big(|X|-m\ind_\Omega\big)\mbf 1_{\{|X|>2m\}}\\
  &\leq 2\big(|X|-m\ind_\Omega\big)\mbf 1_{\left\{|X|>m\right\}}= 2|X-X_m|,
 \end{align*}
 which shows that $\|X\ind_{\{|X|>m\}}\|_{L^\Phi(\CP)}\to 0$ as $m\to \infty$.
\end{proof}

For the sake of clarity, we give the proofs of Lemma~\ref{lem:monclass}, Proposition~\ref{prop:equivalences0}, and Proposition~\ref{prop:equivalences} in advance of Theorem~\ref{thm.super}.

\begin{proof}[Proof of Lemma~\ref{lem:monclass}]
 Let $X\in\CH$ and $c\in\R$. Consider $Y_k:=k(X-c\ind_\Omega)^+\wedge\ind_\Omega\in\CH$, $k\in\N$. The sequence $(Y_k)_{k\in\N}$ is nondecreasing and satisfies $0\peq Y_k\peq \ind_\Omega$. By monotone convergence, 
\begin{equation}\label{eq:1}\mu(\{X>c\})=\sup_{k\in\N}\int Y_k\,{\rm d}\mu=\lim_{k\to\infty}\int Y_k\,{\rm d}\mu\end{equation}
holds for all $\mu\in\ca_+(\CP)$. Moreover, by Dedekind $\sigma$-completeness of $\CH$, $U:=\sup_{k\in\N}Y_k$ exists and lies in $\CH_+$. \textit{A priori}, $\ind_{\{X>c\}}\peq U$ has to hold. Moreover, one can show that $U=(nU)\wedge\ind_\Omega$ holds for all $n\in\N$. Hence, there is an event $B\in\CF$ such that $\ind_B=U$ in $\CH$.  For each $\mu\in\ca_+^\sigma(\CH)$, 
\begin{equation}\label{eq:4}\lim_{k\to\infty}\int Y_k\,{\rm d}\mu=\mu(B).\end{equation}
Equations \eqref{eq:1} and \eqref{eq:4} together with $\CP\approx \ca_+^\sigma(\CH)$ now imply that $\ind_{\{X>c\}}=U\in\CH$, that is, for every $X\in\CH$, $f\in X$, and $c\in\R$, the equivalence class generated by $\ind_{\{f>c\}}$ lies in $\CH$. 
At last, consider the $\pi$-system $\Pi:=\{\{f>c\}\, |\,  X\in\CH,\,f\in X,\,c\in\R\}$, which generates $\CF$ and is a subset of 
\[\Lambda:=\{A\in\CF\, |\, \ind_A\in\CH\}.\]
Since $\CH$ is Dedekind $\sigma$-complete and $\CP\approx \ca_+^\sigma(\CH)$, the latter can be shown to be a $\lambda$-system. By Dynkin's Lemma, it follows that $\Lambda=\CF$. We have thus shown that $\CH$ contains all representatives of $\CF$-measurable simple functions. Each $X\in\Linfty$ is the supremum of a countable family of simple functions in $L^\Phi(\CP)$. As $\CH$ is Dedekind $\sigma$-complete and $\CP\approx \ca_+^\sigma(\CH)$, we conclude that $\Linfty\subset \CH$.
\end{proof}

\begin{proof}[Proof of Proposition \ref{prop:equivalences0}]
{As $\CP\subset\ca(\CH)$}, (2) clearly implies (1). In order to see that (1) implies (3), note first that $\Linfty\subset\CH$ holds by Lemma~\ref{lem:monclass}. Now let $X\in L^\Phi(\CP)$, $Y\in\CH$, and assume $0\peq X\peq Y$ holds. The set $\{X\wedge n\ind_\Omega\mid n\in\N\}\subset\CH$ is order bounded above by $Y$ in $\CH$. By Dedekind $\sigma$-completeness, $X^*:=\sup_{n\in\N}X\wedge n\ind_\Omega$ exists in $\CH$ and satisfies $X\peq X^*$ \textit{a priori}. Arguing as in Lemma~\ref{lem:monclass}, one verifies $X=X^*\in\CH$. \\
In order to see that (3) implies (2), we first show that $\CH$ is Dedekind $\sigma$-complete. Let $\mathcal D\subset\CH$ be order bounded from above and countable. Since $L^\Phi(\CP)$ is Dedekind $\sigma$-complete, $U:=\sup\mathcal D$ exists in $L^\Phi(\CP)$. Let $Y\in\CH$ be any upper bound of $\mathcal D$ and $X\in\mathcal D$. Then, $X\peq U\peq Y$. As $\CH$ is an ideal in $L^\Phi(\CP)$, $U\in \CH$ has to hold and we have proved that $\CH$ is Dedekind $\sigma$-complete.\\ 
Now we prove that each $\mu \in \ca(\CH)$ is $\sigma$-order continuous. For condition (i) in Definition~\ref{def:sigmaorder}, we can argue as in the proof of Lemma~\ref{lem:representation00}. For condition (ii), let $(X_n)_{n\in \N}\in \CH$ be a sequence with $X_{n+1}\peq X_n$ for all $n\in \N$ and $\inf_{n\in \N} X_n=0$ in $\CH$. By \cite[Theorem 1.35]{AliBurk2}, $\inf_{n\in\N}X_n=0$ holds in $L^\Phi(\CP)$, which is equivalent to $\inf_{\P\in\CP}\P(X_n\downarrow 0)=1$.
Moreover, by definition of $\ca(\CH)$, $\int X_1\,{\rm d}|\mu|<\infty$. Dominated convergence yields 
$$\lim_{n\to\infty}\int X_n\,{\rm d}\mu=\lim_{n\to\infty}\int\Big(\tfrac{{\rm d}\mu}{{\rm d}|\mu|}X_n\Big)\,{\rm d}|\mu|=0.$$
Now assume that, additionally, $\CH\subset\Linfty$. If (3) holds, $\CH$ is an ideal containing the equivalence class of $\ind_\Omega$ and must therefore also be a superset of $\Linfty$. Trivially, (4) implies (3), and the proof is complete.  
\end{proof}

\begin{proof}[Proof of Proposition \ref{prop:equivalences}]
The equivalence of (1)--(3) follows directly from Proposition~\ref{prop:equivalences0} up to two additional observations: $\CC$ is a Banach lattice by Proposition~\ref{lem:representation0}, and therefore each element of $\ca^\sigma(\CC)$ is a continuous linear functional by \cite[Proposition 1.3.7]{MeyNie}.
If $\CH\subset\Linfty$, $\CC\subset\cl(\Linfty)$ must hold, and the converse inclusion is a direct consequence under (3). (4) implies (3) because $\Linfty$ is an ideal and norm closures of ideals in Banach lattices remain ideals (\cite[Proposition 1.2.3(iii)]{MeyNie}). 
\end{proof}

\begin{proof}[Proof of Theorem \ref{thm.super}]
(1) is equivalent to (2): Theorem~\ref{thm.main} provides a strictly positive linear functional in the present situation. Hence, the equivalence of (1) and (2) follows with \cite[Lemma A.3]{nend20}.\\
(1) implies (3): Under assumption (1), $\CC$ is a separable and Dedekind $\sigma$-complete Banach lattice. From \cite[Corollary 4.52]{AliBurk2}, we deduce that $\Norm_{L^\Phi(\CP)}$ is $\sigma$-order continuous on $\CC$. Now, in view of Lemma \ref{lem:sigmaorder} and \cite[Proposition 1.3.7]{MeyNie}, $\sigma$-order continuity of the norm on $\CC$ shows
$$\CC^*=\ca(\CC)=\ca^\sigma(\CC).$$
In particular, each $\P\in\CP$ satisfies $\P\in\ca_+^\sigma(\CC)$.
Lemma~\ref{lem:monclass} implies $\Linfty\subset\CC$, which entails that, for all $X\in\CC$, $|X|\ind_{\{|X|\le n\}}\uparrow|X|$ as $n\to\infty$, both in order and in norm. This proves $\CC=\cl(\Linfty)$, which is (3). 

(3) always implies (4). 

(4) implies (1): This has been demonstrated already in the proof of Proposition~\ref{prop:equivalences}.

(1)--(4) implies (5): Note that the equivalent assertions (1)--(4) have already been demonstrated to imply $\CC^*=\ca(\CC)\supset\ca(L^\Phi(\CP)$. For the converse inclusion $\ca(\CC)\subset\ca(L^\Phi(\CP))$, {we can argue as in \eqref{eq:approximation} to see that every $\mu\in\ca(\CC)$ satisfies 
\[\int |X|\,{\rm d}|\mu|\le\|\mu\|_{\CC^*}\|X\|_{L^\Phi(\CP)},\quad \text{for all }X\in L^\Phi(\CP).\]
This means that $|\mu|$ (or equivalently, $\mu$) lies in $\ca(L^\Phi(\CP))$.} Finally, let $\P^*\in \ca\big(L^\Phi(\CP))$ as in Theorem \ref{thm.main} and let $A\in\CF$. $\ind_A\in \CC$ is implied by (3), and it follows that $\ca(L^\Phi(\CP))\approx \CP\approx \P^*$.\\ 
In order to see that the densities of measures in the unit ball of $\ca(\CC)$ form a weakly compact subset of $L^1(\P^*)$, note that \eqref{eq:approximation} admits the representation
\[\{\mu\in\ca(\CC)\mid \|\mu\|_{\CC^*}\le 1\}=\big\{\mu\in\ca(\CP)\,\big|\,\forall X\in L^\infty(\P^*)\colon \big|\smallint X\,{\rm d}\mu\big|\le \|X\|_{L^\Phi(\CP)}\big\}.\]
The right-hand side is clearly weakly closed in $L^1(\P^*)$.\\
Now we consider a sequence $(A_n)_{n\in\N}$ such that 
\[\P^*(A_n)\le 2^{-n}\text{ and }\|\ind_{A_n}\|_{L^\Phi(\CP)}\ge \tfrac 1 2 \sup\{\|\ind_B\|_{L^\Phi(\CP)}\mid B\in\CF,\,\P^*(B)\le 2^{-n}\}.\]
$(\ind_{A_n})_{n\in\N}$ is a sequence in $\CC$ converging to 0 in order, and $\lim_{n\to\infty}\|\ind_{A_n}\|_{L^\Phi(\CP)}=0$.\footnote{~More precisely, set $B_n:=\bigcup_{k\ge n}A_k$, a decreasing sequence of events. As $\P^*(B_n)\downarrow 0$, $\ind_{B_n}\downarrow 0$ holds w.r.t.\ the $\CP$-q.s.\ order in $\CC$. It remains to note that $\ind_{A_n}\peq\ind_{B_n}$, $n\in\N$.}  Set $\mf B$ to be the set of all $\mu\in\ca_+(\CC)$ with $\|\mu\|_{\CC^*}\le 1$. We obtain
\[\sup\{\mu(B)\mid \mu\in\mf B,\,B\in\CF,\,\P^*(B)\le 2^{-n}\}=\sup\{\|\ind_B\|_{L^\Phi(\CP)}\mid B\in\CF,\,\P^*(B)\le 2^{-n}\}\to 0,\quad n\to\infty.\]
This shows that, for all $\eps>0$, there is $\delta>0$ such that $\P^*(B)\le \delta$ implies $\mu(B)\le\eps$, no matter the choice of $\mu\in\mf B$. Moreover, $\mf B$ is bounded in total variation. By \cite[Theorem 4.7.25]{Bogachev}, $\mf B$ and thus also the unit ball of $\ca(\CC)^*$ is weakly compact in $L^1(\P^*)$. This completes the verification of (5).

(5) implies (3): Under assumption (5), the unit ball of $\ca\big(L^\Phi(\CP)\big)$, which is sufficient to determine $\Norm_{L^\Phi(\CP)}$ on all of $L^\Phi(\CP)$, can be identified with a weakly compact subset $\mathcal Z\subset L^1(\P^*)$. Each $X\in\Linfty=L^\infty(\P^*)$ can be identified with a (linear) continuous function on $\mathcal Z$, and if $\Linfty \ni X_n\downarrow 0$ $\CP$-q.s.\ (or $\P^*$-a.s.), the associated sequence of functions converges pointwise to $0$ on $\mathcal Z$. As this pointwise convergence must be uniform, $\sigma$-order continuity of $\Norm_{L^\Phi(\CP)}$ on $\Linfty$ follows.

We now observe
\begin{itemize}
\item For each $X\in\CC$ and each $c\in\R$, the sequence $Y_k:=k(X-c\ind_\Omega)^+\wedge\ind_\Omega\in\CC\cap\Linfty$, $k\in\N$, satisfies $Y_k\uparrow \ind_{\{X>c\}}$ in $\Linfty$.
\item For each increasing sequence $(A_n)_{n\in\N}$ of events in $\Lambda:=\{A\in\CF\mid \ind_A\in\CC\}$, $\ind_{A_n}\uparrow \ind_{\bigcup_{k\in\N}A_k}$ holds in $\Linfty$. 
\end{itemize}
Arguing as in Lemma~\ref{lem:monclass} and using $\sigma$-order continuity of the norm as well as closedness of $\CC$ shows $\CC\cap\Linfty=\Linfty$, i.e., $\mathcal M:=\cl(\Linfty)\subset\CC$. 

Towards a contradiction, assume that we can find $X\in\CC\setminus \mathcal M$. Then, there is a measure $0\neq\mu\in\CC^*=\ca(L^\Phi(\CP))$ such that 
\[\mu|_{\mathcal M}\equiv 0\quad\text{and}\quad \int X\,{\rm d}\mu \neq 0.\]
This however would mean $\mu|_{\Linfty}\equiv 0$, which is impossible. $\CC\subset\mathcal M$ follows. 

We have already proved (6) above. 

For (7), assume that condition \eqref{eq:growth} holds. Then, there exist $a>0$ and $b\geq 0$ such that $\phi_\P(x)\geq ax-b$ for all $\P\in \CP$ and $x\in [0,\infty)$. By~\eqref{eq:normP}, 
\begin{equation}\label{eq:supnormP}
\sup_{\P\in\CP}a(1+b)^{-1}\E_\P[|X|]\leq \|X\|_{L^\Phi(\CP)},\quad \text{for }X\in L^\Phi(\CP).
\end{equation}
The assertion follows with (5), and the proof is complete.
 \end{proof}

\section{Proofs of Section \ref{sec.applications}}\label{app.C}

\begin{proof}[Proof of Proposition~\ref{prop:density}]
As each subset of a separable normed space is separable itself, we can w.l.o.g.\ consider the maximal case $\CH=C_b$.
By Theorem \ref{thm:L1}, there exists some constant $\kappa>0$ such that 
$$\|X\|_{L^\Phi(\CP)}\leq \kappa \|X\|_{L^\infty(\CP)},\quad\text{for all }X\in C_b.$$
Let $d$ be a metric consistent with the topology on $\Om$, and $(\om_n)_{n\in \N}$ be dense in $\Omega$. For $m,n\in \N$ and $\om\in \Om$, let $X_{m,n}(\om):=d(\om,\om_n)\wedge m$. The algebra $\mathcal A\subset C_b$ generated by $\{\ind_\Omega\}\cup\{(X_{m,n})\mid m,n\in \N\}$ is separable and separates the points of each compact set $K\subset \Omega$. We show that the separable set
\[
 \mathcal M:=\big\{(X_0\wedge m\ind_\Omega)\vee(-m\ind_\Omega) \,\big|\, X_0\in \mathcal A, \, m\ge 0\big\} 
\]
is dense in $\CC$. To this end, let $X\in C_b$, $\ep>0$, and $K\subset \Om$ compact with $$\|\ind_{\Om\setminus K}\|_{L^\Phi(\CP)}<\frac{\ep}{2(1+2\|X\|_\infty)}.$$
 By the Stone-Weierstrass Theorem, there exists some $Y\in \mathcal M$ with $\|Y\|_{L^\infty(\CP)}\leq 1+\|X\|_{L^\infty(\CP)}$ and
 \[
  \big\|(X-Y)\ind_K\big\|_{L^\infty(\CP)}<\frac{\ep}{2\kappa}.
 \]
 Hence,
 \begin{align*}
  \|X-Y\|_{L^\Phi(\CP)}&\leq \big\|(X-Y)\ind_K\big\|_{L^\Phi(\CP)}+\big\|(X-Y)\ind_{\Om\setminus K}\big\|_{L^\Phi(\CP)}\\
  &\leq \kappa \big\|(X-Y)\ind_K\big\|_{L^\infty(\CP)} + \big(\|X\|_{\Linfty}+\|Y\|_{\Linfty}\big) \|\ind_{\Om\setminus K}\|_{L^\Phi(\CP)}< \ep.
 \end{align*}
\end{proof}

\begin{proof}[Proof of Lemma~\ref{lem:density}]
(1) trivially implies \eqref{eq:compact}. 

Under condition (2), let $\ep>0$ and set $t:=\ep^{-1}$. Choose $K\subset\Omega$ compact with $$\phi_\Max(t)\sup_{\P\in\CP_t}\P(\Omega\setminus K)\le 1.$$ Then,
\[\sup_{\P\in\CP_t}\E_\P[\phi_\P(t\ind_{\Omega\setminus K})]\le\sup_{\P\in\CP_t}\E_\P[\phi_\Max(t\ind_{\Omega\setminus K})]=\phi_\Max(t)\sup_{\P\in\CP_t}\P(\Omega\setminus K)\le 1.\]
Moreover, 
\[\sup_{\P\in\CP\setminus\CP_t}\E_\P[\phi_\P(t\ind_{\Omega\setminus K})]\le 1.\]
This entails $\|\ind_{\Omega\setminus K}\|_{L^\Phi(\CP)}\le \ep$.

Condition (3) is a special case of condition (2).

Suppose now that \eqref{eq:growth} and \eqref{eq:compact} hold. Then, by \eqref{eq:growth}, there exist $a>0$ and $b\ge 0$ such that $ax-b\leq \inf_{\P\in \CP}\phi_\P(x)$ for all $x\geq 0$.
By~\eqref{eq:supnormP}, $\CP$ is a bounded subset of $L^\Phi(\CP)^*$, and we have 
\[\sup_{\P\in\CP}\E_\P[|X|]\le \frac{1+b}{a}\|X\|_{L^\Phi(\CP)}\quad\text{for all }X\in L^\Phi(\CP).\]
Replacing $X$ by $\ind_{\Omega\setminus K}$ for suitable compacts $K\subset\Omega$ immediately yields tightness of $\CP$. 
\end{proof}

\begin{proof}[Proof of Corollary~\ref{cor:representation}]

 Let $(X_n)_{n\in \N}\subset C_b$ with $X_n\downarrow 0$ as $n\to \infty$ and $\al>0$. Then, $\phi_\Max(\al X_n)\in C_b$ for all $n\in \N$ with $\phi_\Max(\al X_n)\downarrow 0$ as $n\to \infty$. Since $\CP$ is weakly compact, \cite[Corollary 33]{Peng} implies
 \[
  \lim_{n\to\infty}\sup_{\P\in \CP}\E_\P[\phi_\P(\al |X_n|)]=0\quad \text{as }n\to \infty.
 \]
 This suffices to conclude $\lim_{n\to\infty}\|X_n\|_{L^\Phi(\CP)}=0$. 
  To $\ell\in \CC^*$ positive we can thus apply \cite[Theorem 7.8.1]{Bogachev} as in the proof Lemma~\ref{lem:sigmaorder} to obtain $\ell=\mu$ for a unique measure $\mu\in \ca(\CP)_+$. 
 By standard arguments, one can extend this identity to the closure $\CC$ of $C_b$, which also yields $\mu\in\ca(\CC)$. To this end, one uses the observation from the proof of Proposition \ref{prop:Banach} that every norm convergent sequence contains a $\CP$-q.s.\ convergent subsequence. 
 \end{proof}

\begin{proof}[Proof of Proposition~\ref{prop:option}]
We first prove that $\CH_X$ is a sublattice of $L^\Phi(\CP)$. The latter space is Dedekind $\sigma$-complete and thus also uniformly complete (\cite[Proposition 1.1.8]{MeyNie}). As such, we may replicate the argument in the proof of \cite[Theorem 3.1]{Options}. 
$\big(\CC_X,\peq,\Norm_{L^\Phi(\CP)}\big)$ is a Banach lattice by Proposition~\ref{lem:representation0}. 

Now, the span of the countable set $\{\ind_\Omega\}\cup\{(X-k\ind_\Omega)^+\mid k\tn{ rational}\}$ over the rational numbers lies dense in $\CH_X$, whence separability of $\CH_X$ and its norm closure $\CC_X$ follow. 
\end{proof}

\begin{proof}[Proof of Corollary~\ref{cor:ae}]
Note that $J_{\P^*}(L^\Phi(\CP))\subset L^1(\P^*)$ is an ideal on which $\P^*$ acts as a strictly positive bounded linear functional. The assertion thus follows directly from \cite[Corollary 3.2(b)]{Options}. 
\end{proof}
 
\end{appendix}

\bibliographystyle{abbrv}


\end{document}